\documentclass[preprint]{imsart}

\usepackage[T1]{fontenc}
\usepackage[latin1]{inputenc}
\usepackage[english]{babel}
\usepackage{amsmath,amsthm,amssymb,amsfonts,epic,latexsym,hyperref,bbm}
\usepackage{mathabx}
\usepackage[dvips]{graphicx}
\usepackage{cancel}
\usepackage{lipsum}
\usepackage{rotating}
\usepackage{dsfont}
\usepackage{geometry}
\usepackage{stmaryrd}
\usepackage{ifthen}
\usepackage[nottoc]{tocbibind}
\usepackage{verbatim}
\usepackage{tabularx}
\usepackage{framed}
\usepackage{longtable}
\usepackage{mathtools}
\usepackage[shortlabels]{enumitem}
\usepackage[usenames, dvipsnames]{xcolor}

\newcommand{\norme}[1]{|\!| #1|\!|}
\def \E{I\!\!E}
\def \P{I\!\!P}
\def \R{{I\!\!R}}

\newcommand{\N}{\mathbb{N}}

\newcommand{\D}{\mathcal{D}}
\newcommand{\B}{\mathcal{B}}
\newcommand{\cP}{\mathcal{P}}
\newcommand{\indiq}{{\mathbbm{1}}}
\newcommand{\law}[1]{{\mathcal{L}( #1 )}}
\newcommand{\ap}{{\alpha_+}}
\newcommand{\am}{{\alpha_-}}
\newcommand{\convlaw}{{\overset{\mathcal{L}}{\to}}}

\theoremstyle{plain}
\newtheorem{theorem}{Theorem}[section]
\newtheorem{proposition}[theorem]{Proposition}
\newtheorem{corollary}[theorem]{Corollary}
\newtheorem{assumption}{Assumption}
\newtheorem{lemma}[theorem]{Lemma}

\newtheorem{definition}{Definition}
\newtheorem{example}{Example}

\usepackage{tikz}
\usepackage{atbegshi}

\usepackage{graphicx}
\usepackage{blindtext}
\usepackage{titlesec}

\begin{document}

\begin{frontmatter}

\AtBeginShipoutFirst{%
    \begin{tikzpicture}
         {   \includegraphics[width=4cm]{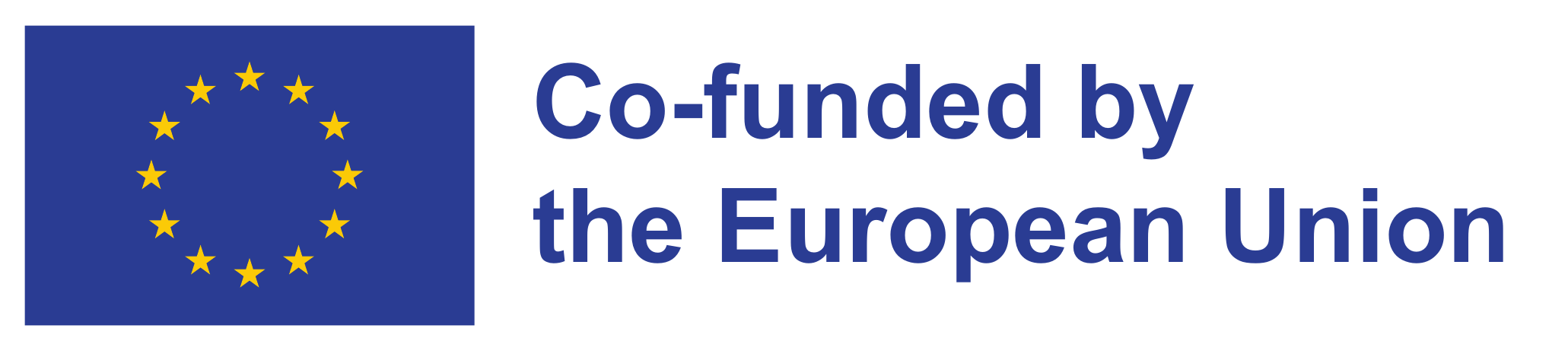} 
        };
    \end{tikzpicture}%
}

\title{Weak conditional propagation of chaos for systems of interacting particles with nearly stable jumps}

\runtitle{Systems of particles with nearly stable jumps}

\begin{aug}

\author{\fnms{Eva} \snm{L\"ocherbach}\thanksref{m1}\ead[label=e4]{eva.loecherbach@polytechnique.edu}},
\author{\fnms{Dasha}
\snm{Loukianova}\thanksref{m2}\ead[label=e5]{dasha.loukianova@univ-evry.fr}}
\author{\fnms{Elisa} 
\snm{Marini}\thanksref{m1}\ead[label=e6]{marini@ceremade.dauphine.fr}}

\address{\thanksmark{m1}CMAP, Ecole Polytechnique, Institut Polytechnique de Paris, 91120 Palaiseau
  \thanksmark{m2}Laboratoire de Math\'ematiques et Mod\'elisation d'\'Evry,  Universit\'e
  d'\'Evry Val d'Essonne, UMR CNRS 8071 
  }

\runauthor{E. L\"ocherbach and D. Loukianova and E. Marini}

 \end{aug}

\begin{abstract}
We consider a system of $N$ interacting particles described by stochastic differential equations driven by Poisson random measures, where the coefficients depend on the empirical measure of the system. 
Each particle jumps at a rate that depends on its position.
Moreover, each  jump of a particle $i$ triggers \textit{simultaneous \lq\lq collateral\rq\rq\, jumps} of all the other particles $j\neq i$: their positions are changed by a random amount $ u/N^{1/\alpha},$ where $u$ is \lq\lq nearly stable\rq\rq, i.e. belongs to the domain of attraction of an $\alpha$-stable law, and where $0 <  \alpha < 2.$ 
We prove convergence in law, in the Skorokhod space, of this system to a limiting infinite exchangeable system of conditional McKean-Vlasov SDEs driven by a common stable process. This stable process emerges from the simultaneous jump term via the stable central limit theorem.  
In the limit, the particles are independent and identically distributed conditionally on this common noise. That is, the $N$-particle system exhibits the conditional propagation of chaos property.  
To establish this convergence result, we develop an approach based on semimartingale convergence that is new in the propagation of chaos setting. Moreover, due to the scaling $N^{1/\alpha}$, and the fact that nearly $\alpha$-stable laws admit moments only of order strictly less than $\alpha,$ proving tightness requires a novel argument based on representing the sum of collateral jumps as a time-changed random walk.  

\end{abstract}

 
\end{frontmatter}

\section{Introduction}

In this paper we study the weak convergence as $N\to \infty$ of the system of $N$ interacting particles $X^N = (X^{N,i})_{i=1}^N$ given by the  SDE
\begin{multline}\label{eq:finite:system:intro}
	X^{N, i}_t =  X^{i}_0 +   \int_0^t  b(X^{N, i}_s , \mu_s^N )  ds +  \int_{[0,t]\times\R_+ \times \R } \psi (X^{N, i}_{s-}, \mu_{s-}^N ) \indiq_{ \{ z \le  f ( X^{N, i}_{s-}) \}} \pi^i (ds,dz, du ) \\
	+ \frac{1}{{N}^{ 1/\alpha} }\sum_{ j \neq i } \int_{[0,t]\times\R_+\times\R  }u \indiq_{ \{ z \le  f ( X^{N, j}_{s-}) \}} \pi^j (ds,dz,du) , \qquad t \geq 0, 1 \le i \le N. 
\end{multline} 
In the above equation,  $ \mu_t^N = \frac1N \sum_{i=1}^N \delta_{X_t^{N, i } } $ is the empirical measure of the system at time $t$, $(\pi^{i})_{i=1,\ldots,N}$ is a family of i.i.d. Poisson random measures on $\R_+ \times \R_+ \times \R$ with intensity $ds dz \nu(du),$  $\nu$ is a law belonging to the domain of attraction of a strictly stable law of index $\alpha\in (0,2)\setminus\{1\}  $ and $b,\psi, f$ are bounded and Lipschitz continuous real functions. 

The dynamics of the system \eqref{eq:finite:system:intro} can be summarized as follows:
\begin{enumerate}
	\item Any particle $i$ jumps at a rate $f$ depending on its own position.     
	When such a jump happens, we say that particle $i$ performs a \textit{main jump}, and its position changes by the value of the function $\psi$ evaluated in the position of the particle $i$ just before jump. This jump height may also depend on the current state of the empirical measure of the whole system. 
	\item At the same time, the main jump of particle $i$ triggers \textit{simultaneous jumps} of all the other particles $j\neq i$: when some particle $i$ has a main jump, the positions of all the other particles $j\neq i$ in the system are changed by the same random amount $u$ sampled from $\nu$ and rescaled by $N^{1/\alpha}$. Borrowing the terminology coined in \cite{andreis_mckeanvlasov_2018}, we say that the other particles perform a \textit{collateral jump}.  
	\item In between consecutive jumps, each particle follows a deterministic flow with drift $b$, which may also depend on the empirical measure of the system. \end{enumerate}

The interest in interacting particle systems whose components can jump simultaneously is motivated for instance by neuroscience applications, see e.g. \cite{Brillinger, Cessac, DGLP, GLP, Touboul, FL}. 
In this context, collateral jumps correspond to the synaptic weight of a neuron on its postsynaptic partners, and main jumps to the hyperpolarization of a neuron after a spike.  

Models similar to \eqref{eq:finite:system:intro} appeared in \cite{DGLP, FL}, where simultaneous jumps were rescaled by $1/N$, leading, in the infinite-volume limit, to the disappearance of the inter-particle interactions.
In this case the limiting  system is an infinite i.i.d. collection of McKean-Vlasov SDEs depending on the limiting empirical measure, which is deterministic and coincides with the law of any particle in the limiting  system.  
This phenomenon is known as \textit{propagation of chaos}. It has been extensively studied during the last decades 
for various classes of SDEs, in particular in the Brownian and/or L\' evy-driven setting 
\cite{chaintron:diez,sznitman, andreis_mckeanvlasov_2018, graham92, graham08}, and, at least under suitable regularity assumptions on the coefficients, the theory is by now well developed. 

Models in which simultaneous jumps are not linearly scaled in $N$ were first studied in \cite{ELL1, ELLPTRF, ELL2}. There, the authors considered systems similar to \eqref{eq:finite:system:intro}, but where the law $\nu$ admits a finite second moment and where simultaneous jumps are rescaled by $1/\sqrt N$. They proved that, in the infinite-population limit, the simultaneous jump term gives rise to a common noise, namely, a common Brownian motion, and that particles in the limit are conditionally i.i.d. given this common noise. The limiting system is an infinite exchangeable system 
composed of conditional McKean-Vlasov SDEs depending on the limiting empirical measure. This limiting empirical measure coincides with the conditional law of any limit particle given the common noise and also the directing measure of the limit system. 
This phenomenon is called the \textit{conditional propagation of chaos property}. 

Note that there is a large literature on conditional propagation of chaos in settings where the common noise is already present at the level of the finite particle system (see e.g. \cite{coghi_propagation_2016, carmona_mean_2016, dermoune_propagation_2003}). To the best of our knowledge, the situation where the common noise only appears in the limiting dynamics is new and has only been studied {in previous papers by the authors}  \cite{ELL1, ELLPTRF, ELL2, LLM, dasha-eva}.

In the present paper we prove conditional propagation of chaos for system \eqref{eq:finite:system:intro} in a weak framework.  Namely, we prove  the convergence in distribution, in the Skorokhod path space $D(\R_+, \R)^{\otimes \infty}$, of system \eqref{eq:finite:system:intro} to an infinite-exchangeable system $ (\bar X^i )_{ i \geq 1 } $  obeying the following conditional McKean-Vlasov equation  
\begin{multline}\label{eq:limitsystem}
\bar X^{i}_t =  X^{i}_0 +   \int_0^t  b(\bar X^{ i}_s , \bar \mu_s )  ds +  \int_{[0,t]\times\R_+ } \psi (\bar X^{ i}_{s-}, \bar \mu_{s-} )  \indiq_{ \{ z \le  f ( \bar X^{ i}_{s-}) \}} \bar \pi^i (ds,dz) 
\\
+ \int_{[0,t] } \left( \bar \mu_{s-}(f)  \right)^{1/ \alpha}   d S_s^\alpha ,  \qquad i \geq 1 .
\end{multline} 
Here $ (\bar \pi^i)_{ i \geq 1 }$ is a family of i.i.d. Poisson random measures on $ \R_+ \times \R_+$ with Lebesgue intensity,  $S^\alpha$ a strictly $\alpha$-stable process independent of $ (\bar \pi^i)_{ i \geq 1 }$ and of the i.i.d. initial values $ (X^i_0)_{i \geq 1}.$ Finally $ \bar \mu_s = {\mathcal L} ( \bar X^1_s | S^\alpha_u, u \le s )$ is the conditional law of 
any particle given the common $\alpha$-stable noise, and also the 
directing measure of the exchangeable system \eqref{eq:limitsystem}
 (see \cite{aldous} Definition (2.6) and Proposition \ref{thm:dir_m} below).   

The main difficulty of the present model, compared with \cite{ELL1, ELLPTRF, ELL2}, is that both the $\alpha$-stable process $S^\alpha$ appearing in the limiting system and the law $\nu$ of the sizes of the simultaneous jumps in \eqref{eq:finite:system:intro} admit finite moments only of order strictly smaller than $\alpha$.  
In our previous paper \cite{LLM} 
we proved the strong well-posedness of \eqref{eq:limitsystem}. We also provided,  for any fixed $t > 0, $ a coupled construction of $ X_t^{N, i } $ and $ \bar X^i_t ,$ on the same probability space, and we obtained an explicit control of $ \E ( d (  X_t^{N, i },  \bar X^i_t ) ), $ for any fixed $ t \geq 0,$ where $ d( x, y) $ is a convenient distance on $ \R, $ depending on $ \alpha.$ 
 
In the present paper we propose a general approach to show the weak convergence of \eqref{eq:finite:system:intro} to \eqref{eq:limitsystem} on the path space $ D ( \R_+, \R),$ not only for fixed time marginals. We believe that this general approach can be applied to prove pathwise weak conditional propagation of chaos for a much larger class of interacting particle systems given by SDEs.
 
The main result of the 
present paper is Theorem \ref{theo:main}, which states that the sequence of empirical measures $(\mu^N)_N$ of \eqref{eq:finite:system:intro} on the path space $D(\R_+, \R)$ converges in law to the directing measure $\bar\mu$ of the limit system \eqref{eq:limitsystem}. 
This implies (Proposition (7.20) in \cite{aldous}) that
the law of the system   ${(X^{N,i})_{i\leq N}},$ completed by $X^{N,i}=0$ for $i>N,$ on  the Skorokhod path space $D(\R_+, \R)^{\otimes \infty}$ converges to the law of \eqref{eq:limitsystem}.  

The approach we propose to prove such a convergence result 
is, to the best of our knowledge, new for  mean field interacting systems, and relies on the theory of weak convergence of semimartingales via the convergence of their characteristics as developed in \cite{JS}. 
It consists of four steps. 
\begin{itemize}
\item Proving the tightness of the sequence $\mu^N$ in the space of probability measures ${\cal {P}}(D(\R_+, \R))$ on the Skorokhod space $D(\R_+, \R)$, which implies the tightness of $(\mu^N, X^{N,1},\ldots, X^{N,m})$ in ${\cal P}(D(\R_+, \R))\times D(\R_+, \R)^m,$  for any fixed $m\in\N^*.$
\item Considering any convergent subsequence of the semimartingale $(X^{N,1},X^{N,2} )$ (due to the exchangeability it is sufficient to consider  the limit for a couple) and identifying the limit of its characteristics. The weak limit of this subsequence will be a semimartingale $( X^{1}, X^{2})$ with its associated limiting characteristics. 
\item Identifying--using representation theorems \cite{IW}--$( X^{1}, X^{2})$ as a weak solution of  \eqref{eq:limitsystem}, where $\bar \mu$  is the weak limit along a converging subsequence of $\mu^N$ and hence the directing measure of \eqref{eq:limitsystem}. 
\item 
Identifying $\bar \mu$ as the conditional law $ \bar \mu_s = {\mathcal L} ( \bar X^1_s | S^\alpha_u, u \le s )$. 
\end{itemize}
On the contrary to the case where the simultaneous jump distribution $\nu$ possesses a second moment, showing the tightness of the sequence $\mu^N$ is a delicate step here. 
We do this in Section \ref{sec:tightnes_muN}. The main point is that we are able to prove that the sum of the collateral jumps 
\begin{equation}\label{eq:rw_collja}
\frac{1}{{N}^{ 1/\alpha} }\sum_{ j =1}^N \int_{[0,t]\times\R_+\times\R  } u \indiq_{ \{ z \le  f ( X^{N, j}_{s-}) \}} \pi^j (ds,dz,du)
\end{equation}
converges in law to a time-changed stable process 
$$
 S^\alpha_\cdot \circ A_t, 
$$
where $S^\alpha_\cdot$ is a strictly stable process and $A_\cdot = \int_0^\cdot \bar\mu_s(f) ds$ is the integrated mean jump intensity of the limit system.

This holds since \eqref{eq:rw_collja} can be rewritten as a time-changed compound Poisson process with i.i.d. $\nu$-distributed summands $(U_l)_l$, 
$$
\left( \frac{1}{N^{1/\alpha}}\sum_{l=1}^{\tilde P_{N\cdot}} U_l  \right)  \circ \int_0^t  \mu^N_{s}(f) ds ,
$$
where $\tilde P$ is a homogeneous unit-rate Poisson process. Given the boundedness of the intensity $f$, the sequence of the random time changes $(\int_0^t  \mu^N_{s}(f) ds)_N $ is tight in the Skorokhod space.  The stable central limit theorem together with continuity arguments then implies convergence. Proposition \ref{prop:convergence_J} states this result rigorously and explains the emergence in the limit of the $\alpha$-stable process. Note that in the toy case where $\psi$ and $b$ are identically zero, Proposition \ref{prop:convergence_J} would be sufficient to establish the convergence of the finite system to its limit. 

While Proposition \ref{prop:convergence_J} expresses the main idea and intuition about the limiting system  \eqref{eq:limitsystem}, in the general case it is not enough to prove the convergence of \eqref{eq:finite:system:intro} to  \eqref{eq:limitsystem}. In particular, we have to identify the joint law of the main jumps and of the collateral jump part in the limit, and to do so we rely on the semimartingale techniques mentioned above.

The paper is organized as follows. In Section \ref{sec:model_ass_mainres}, we present the model, detail our assumptions and state the main results. 
In Section \ref{sec:tightnes_muN}, we prove tightness of $( (X^{N,i})_i, \mu^N)_N$. Section \ref{sec:5} is devoted to the proof of our main result and is structured as follows.  
In Section \ref{sec:common_convergence_results} we consider any weakly convergent subsequence of $(X^N, \mu^N)$ and show that its limit in law is $(X, \mu)$, where $\mu$ is the limit of $(\mu^N)_N$ along a fixed convergent subsequence and where $X = (X^{i})_i$ is an infinite-exchangeable system directed by $\mu$. 
This allows to pass to the limit in distribution in equation \eqref{eq:finite:system:intro} and to identify the SDEs satisfied by $X$--which is done in Subsection \ref{sec:passage_to_lim:alpha<1} for $\alpha<1$ and in Subsection \ref{sec:passage_to_lim:alpha>1} for $\alpha>1$. These SDEs have the structure of \eqref{eq:limitsystem}, but the form of the directing measure $ \mu$ as conditional law given $ S^\alpha $ still has to be proved. This proof is provided in Section \ref{sec:uniqueness}. Here we identify the limit law and conclude by proving that the system of SDEs obtained for the process $X$ coincides with the system \eqref{eq:limitsystem}. The fact that \eqref{eq:limitsystem} possesses a unique strong solution (\cite{LLM}) is crucial here. This allows to conclude that the sequence $(\mu^N)_N$ converges in law to $\mu=\bar\mu$.  
The Appendix section contains some useful lemmas and continuity results.

\section{Notation}\label{sec:notation}
Throughout this article, we will use the following notation. 

$\R^\ast \coloneqq \R \setminus \{0\} .$

 $ C (\R_+, \R)$ denotes the space of continuous functions, endowed with the topology of uniform convergence on compact sets: for $(\gamma^N)_N \subset C (\R_+, \R), \, \gamma\in C (\R_+, \R)$, as $ N \to \infty, $ $\gamma^N \to \gamma$ if $ \lim_{ N \to \infty } \sup_{t\leq T} |\gamma^N(t) - \gamma(t)| =  0$ for any $T>0.$

$D(\R_+, \R)$ denotes the space of c{\`a}dl{\`a}g functions, endowed with the Skorokhod $J_1$ topology (see \cite{JS} VI.1.b., see also \cite{kallenberg}, Appendix A.5). Moreover, $ \D \coloneqq {\mathcal D} ( \R_+ , \R ) $ denotes the Borel sigma-algebra of $ D( \R_+, \R).$ 

\sloppy $\Lambda \coloneqq \{ \lambda: \R_+ \to \R_+ \, | \, \lambda \text{ is strictly increasing, continuous,} \, \lambda(0)=0, \, \lim_{t\to+\infty}\lambda(t)= +\infty \} \subset C(\R_+, \R)$ denotes the space of time changes needed in the definition of the Skorokhod topology.  

We will often use that, as $ N \to \infty, $  
convergence of $\gamma^N$ to $\gamma$ with respect to the Skorokhod topology occurs if and only if there exists a sequence $\{\lambda^N\}_N \subset \Lambda$ such that for all $ T > 0, $ 
\begin{enumerate}
	\item	$\lim_{N\to +\infty} \sup_{t\leq T} |\lambda^N(t) - t| = 0 ;$
	\item $\lim_{N\to +\infty} \sup_{t\leq T} | \gamma^N(\lambda^N(t)) - \gamma(t) | =  0 ;$
\end{enumerate}
see Lemma A.5.3 of \cite{kallenberg}. 

Often we will restrict the space of time changes and rather consider 

$\Lambda^{\lambda_1, \lambda_2 }  \coloneqq \{ \lambda: \R_+ \to \R_+ \, | \, \lambda \in \Lambda \text{ and } \lambda_1 \le \frac{\lambda(t)-\lambda(s)}{t-s} \le \lambda_2  \; \forall  0 \le  s<t < \infty \} ,$ where $ 0 < \lambda_1 < \lambda_2 $ are two fixed constants, which is the subspace of $ \Lambda $ containing all time changes having a modulus of continuity which is controlled by $ \lambda_1 $ and $ \lambda_2. $ 

$\mathcal{P}(S) $ denotes the space of probability measures on $ (S, {\mathcal S}), $ where $S$ is a generic Polish space and  $ {\mathcal S} $ its  Borel sigma-algebra. 

$\mathcal{P}_1(\R)$ denotes the space of probability measures on $\R$ having a finite first moment. 



${\mathcal C}_1(\R^d)$ (see \cite{JS} VII.2.7) is a convergence-determining class for the weak convergence induced by the set of bounded continuous functions $g: \R^d\to \R$ which are zero around zero. 

$X_n \convlaw Y$ denotes the convergence in distribution of $ X_n$ to $ Y.$

For two probability measures $\nu_1, \nu_2 \in \mathcal{P}_1 (\R),$ the Wasserstein distance of order $1$ between $\nu_1$ and $\nu_2$  is defined as
\[
W_1(\nu_1,\nu_2)=\inf_{\pi\in\Pi(\nu_1,\nu_2)} \int_\R\int_\R |x-y| \pi(dx,dy)  ,
\]
where $\pi$ varies over the set $\Pi(\nu_1,\nu_2)$ of all probability measures on the product space $\R\times \R$ with marginals $\nu_1$ and $\nu_2$. 

For some $ 0 < q < 1, $ we will also consider the distance
\begin{equation*}
 d_{q}(x,y) = |x-y| \wedge |x-y|^{q},
\end{equation*} 
and, for any $\nu_1,\, \nu_2 \in \mathcal{P}_1(\R)$, 
\[ W_{d_{q}}(\nu_1,\nu_2) = \inf_{\pi \in \Pi(\mu_1,\nu_2)} \int_{\R}\int_{\R} d_{q}(x,y) \pi(dx,dy) \]
which is the Wasserstein distance associated with the metric $d_{q}$ (\cite{villani}). 

As for the classical Wasserstein distance, the Kantorovich-Rubinstein duality yields (\cite{villani}, particular case 5.16)
\[W_{d_{q}}(\nu_1,\nu_2) = \sup\{\nu_1(\varphi) - \nu_2(\varphi) : \forall x,\,y\in \R\,\, |\varphi(x) - \varphi(y)| \leq  d_{q}(x,y) \} . \]
Notice that $W_{d_{q}}(\nu_1,\nu_2)\leq W_1(\nu_1,\nu_2)$ for all $\nu_1,\,\nu_2 \in \mathcal{P}_1(\R)$. 

Throughout this article, $ S^\alpha = (S^\alpha_t)_{t \geq 0} $ denotes a strictly $\alpha-$stable L{\'e}vy process given by 
\begin{align}\label{eq:St}
	& S^\alpha_t = \int_{ [0, t ] \times \R^\ast} z M (ds, dz ) , \mbox{ if } \alpha < 1  \\ \nonumber
        & S^\alpha_t = \int_{ [0, t ] \times \R^\ast} z \tilde M (ds, dz ) , \mbox{ if } \alpha >  1 , 
  \end{align}
(see \cite{applebaum}, \cite{sato}).   
Its jump measure $M$ is a Poisson random measure on $ \R_+ \times \R^\ast $ having intensity $ds \nu^\alpha(dz)$, with
\begin{equation}\label{eq:nualpha}
\nu^\alpha(dz) = \frac{a_+}{z^{\alpha+1}}\indiq_{\{z>0\}} dz + \frac{a_-}{|z|^{\alpha+1}}\indiq_{\{z<0\}}dz ,
\end{equation}
where $ a_+, a_- \geq 0 , a_+ + a_- > 0 , $ are some fixed parameters, and $\tilde M(ds,dz) \coloneqq M(ds,dz) - \nu^\alpha(dz) ds$ denotes the compensated Poisson random measure.

\section{Model, assumptions and main results}\label{sec:model_ass_mainres}

\subsection{The finite system}
We consider an interacting particle system $(X^{N,i}_t)_{ t \geq 0 , 1 \le i \le N } $ taking values in $ \R^N $ which is solution of 
\begin{equation}\label{eq:finite:system}
   \left\{\begin{array}{rcl}
    X^{N,i}_t & = &\displaystyle X^{N,i}_0 + \int_0^t b(X^{N,i}_s, \mu^N_s) ds + \int_{[0,t]\times \R_+ \times \R} \psi(X^{N,i}_{s^-}, \mu^N_{s^-}) \indiq_{\{ z \leq f(X^{N,i}_{s^-})\}} \pi^i(ds,dz,du) \\
    && +\displaystyle  \frac{1}{N^{1/\alpha}} \sum_{j\neq i}\int_{[0,t]\times \R_+ \times \R } u \indiq_{\{ z \leq f(X^{N,j}_{s^-}) \}} \pi^j(ds,dz,du), \qquad i=1,\dots, N,\\
  X_0^{N,i}&\sim& \nu_0 , 
  \end{array}\right.
    \end{equation}
for all $t \geq 0, $  where the initial positions $ (X^i_0)_{ i \geq 1}  $ are i.i.d., distributed according to some fixed probability measure $ \nu_0.$ In the above equation, $ \mu^N = \frac1N \sum_{i=1}^N \delta_{X^{N, i} } $  is the empirical measure of the system, $ \mu^N_s = \frac1N \sum_{i=1}^N \delta_{X^{N, i} _s} $ its projection onto time $s,$ and  
$ \pi^j (ds, dz, du ), j \geq 1,$ are i.i.d. Poisson random measures on $ \R_+ \times \R_+ \times \R ,$ independent of the initial positions, having intensity measure $ ds dz \nu ( du ) ,$ where $ \nu $ is a probability measure on $ \R.$ 
Finally, 
 $ \alpha \in ( 0, 2 ) $ is a fixed parameter whose meaning will become clear in the sequel. 

Below we state conditions that ensure that system \eqref{eq:finite:system} admits a unique strong solution.

\subsection{Assumptions}
We will impose the following conditions on the coefficients $b, f $ and $\psi $ of equation \eqref{eq:finite:system}.

\begin{assumption}\label{ass:b}
\begin{enumerate}
\item $b$ is  bounded. 
\item There exists a constant $ C > 0$ such that for every $x,\,y \in \R$ and every $\mu,\,\tilde \mu \in \mathcal{P}_1(\R)$, it holds $ |b(x,\mu) - b(y,\tilde \mu) | \leq C \left( |x-y| +W_{1}(\mu,\tilde \mu )\right).$ 
\end{enumerate}
\end{assumption}

\begin{assumption}\label{ass:f}
\begin{enumerate}
\item $f$ is bounded.
\item $f$ is Lipschitz-continuous.  
\item $f$ is lowerbounded by some strictly positive constant $ \underline f > 0$. 
\end{enumerate}
\end{assumption}

\begin{assumption}\label{ass:psi}
\begin{enumerate}
\item $ \psi$ is bounded. 
\item  There exists a constant $ C > 0$ such that for every $x,\,y \in \R$ and every $\mu,\,\tilde \mu \in \mathcal{P}_1(\R)$, it holds $ |\psi(x,\mu) - \psi(y,\tilde \mu) | \leq C \left( |x-y|  +W_{1}(\mu,\tilde \mu)\right).$ 
\end{enumerate}
\end{assumption}

Recall that each Poisson random measure $ \pi^i  (dt, dz, du ) $ has the same intensity $ dt dz \nu ( du).$ We now give the precise conditions on  $ \nu$  that will be needed in the sequel. To start with, we first recall the definition of domain of attraction of a stable law.  
\begin{definition}[Domain of attraction]\label{def:normal_doa}
A probability law $ \nu $ on $ \R$ belongs to the domain of attraction of a strictly stable law of index $\alpha \in (0,2)$ if for any i.i.d. sequence $(U_n)_{n \geq 1 } $ of real-valued random variables such that $ U_n \sim \nu$ there exist real-valued sequences $ (a_n)_n, (b_n)_n$ such that $ b_n > 0 $ for all $n$ and such that, as $ n \to \infty, $ 
$$ \frac{1}{b_n} \sum_{k=1}^n U_k - a_n \convlaw S^\alpha_1 ,$$
with $S^\alpha_1$ a strictly stable variable of index $\alpha$ as in \eqref{eq:St}.
\end{definition}
By  \cite{feller}, Chapter IX.8 Theorem 1 and the discussion right below this theorem,  $ \nu$ belongs to the domain of attraction of a stable law if and only if its distribution function $ F ( x) \coloneqq \nu ((-\infty, x] ) $ satisfies 
$$ 1 -  F ( x) + F( - x) \sim x^{-\alpha } L( x) , $$
as $ x \to \infty, $ where $L $ varies slowly at infinity, and 
$$ \lim_{x \to \infty} \frac{ 1 - F (x) }{  1 -  F ( x) + F( - x)  } = p, \lim_{x \to \infty} \frac{ F(-x)  }{  1 -  F ( x) + F( - x)  } = 1- p ,$$ 
for some $ p \in [0, 1].$

\begin{assumption}\label{ass:nu}
We suppose that $\nu $ belongs to the domain of attraction of a strictly stable law of index $ \alpha \in ( 0, 2) \setminus \{ 1 \}, $ such that we can choose $ b_n = n^{ 1 / \alpha } .$ \footnote{Compare to (8.13) in Chapter IX.8 of \cite{feller} for possible choices of the norming constants $b_n.$} 
 If $\alpha \in (1,2)$, we additionally assume that $\nu$  is centered. 
\end{assumption}

Notice that this last assumption allows to take $a_n=0$ in Definition \ref{def:normal_doa} for all $\alpha\in (0,2)\setminus\{1\}$ (see for instance \cite{feller} IX.8).

Recall that the initial positions $ (X^i_0)_{ i \geq 1}  $ are i.i.d., distributed according to some fixed probability measure $ \nu_0.$ We assume: 
\begin{assumption}\label{ass:nu0}
$\nu_0 $ admits a finite moment of order $ (2 \alpha ) \vee 1 $ in case $ \alpha < 1 $  and a finite moment of order $p $ for some $p > 2 $ in case  $ 1 < \alpha < 2. $
\end{assumption}

Finally, in the case $ \alpha > 1, $ we always assume that $ \psi \equiv 0 ,$ that is, there are no main jumps present in this case.
This assumption is purely technical and explained by the fact that the natural norm to deal with the main jumps is the $L^1$- norm, whereas the martingale term corresponding to the collateral jumps is naturally controlled in the $L^2$-norm. This issue arises in particular when showing the pathwise uniqueness for the limit system, see \cite{LLM}, Section 3.1.   

Theorem 2.12 of \cite{LLM} implies that under the above Assumptions \ref{ass:b}--\ref{ass:nu0}, the system  \eqref{eq:finite:system} admits a unique strong solution such that for any $ t \geq 0, 1 \le i \le N,$ $ X_t^{N, i } $ has a finite moment of order $p $ for any $ p < \alpha .$

In what follows, we are interested in the large population limit of the system \eqref{eq:finite:system} and study its convergence as $ N \to \infty.$ To do so, we need to impose additional assumptions on $b$ and $ \psi.  $   

\begin{assumption}\label{ass:weak_cont}
\begin{enumerate}
\item 
$b$ and $\psi$ are continuous with respect to the product topology of the standard Euclidean topology on $\R$ and the topology of weak convergence on $\mathcal{P}(\R)$.
\item
 In case $ \alpha < 1, $ we fix some $ 0 < \alpha_- < \alpha $ and we suppose moreover that 
 for every $x,\,y \in \R$ and every $\mu,\,\nu \in \mathcal{P}_1(\R)$, $ |b(x,\mu) - b(y,\nu) | + |\psi(x,\mu) - \psi(y,\nu) | \leq C \left( d_{ \alpha_- }(x,y) +W_{d_{ \alpha_- }}(\mu,\nu)\right).$ 
\end{enumerate}
\end{assumption}  
 
\begin{example}
Any functions $b (x, m)$ and $ \psi (x, m) $ which are given by 
\[
	b(x,m) \coloneqq \int_\R B(x-y) m(dy), \qquad \psi (x, m) \coloneqq \int_\R \Phi ( x- y ) m (dy) ,
\]
with $B, \Phi$ bounded Lipschitz continuous functions, satisfy all the above hypotheses. Indeed, it is easily seen that in this case, $b(x_n, m_n ) $ converges to $b(x,m)$ whenever $x_n \to x $ and $ m_n $ converges weakly to $m,$ and that there exists $C>0$ such that $| b(x,\nu_1) - b(y,\nu_2) | \leq C (|x-y|  + W_1(\nu_1, \nu_2))$. The same arguments apply to $ \psi ( x, m ) .$  In case $ \alpha < 1, $ we have to assume additionally that $ B$ and $ \Phi$ are Lipschitz continuous with respect to $ d_{\alpha_-}.$  
\end{example}

\subsection{The limit system}
The limit system associated with our particle system is the infinite-exchangeable system 
$ ( \bar X^i)_{ i \geq 1 } $ which is solution to
	\begin{equation}\label{eq:limit:system}
\left\{\begin{array}{rcl}
	\bar X^i_t & = &\displaystyle  \bar X^i_0 + \int_0^t b(\bar X^i_s, \bar\mu_s) ds + \int_{[0,t]\times \R_+} \psi(\bar X^{i}_{s^-}, \bar\mu_{s^-}) \indiq_{\{ z\leq f(\bar X^{i}_{s^-})\} } \bar\pi^i(ds,dz)\\
	&& +\displaystyle  \int \bar\mu^{1/\alpha}_{s^-}(f) d S^\alpha_s ,\qquad i\geq 1 ,\\ 
 \bar \mu &=& {\mathcal L}( \bar X^i | S^\alpha ) ,
\end{array}\right.
\end{equation}
for all $ t \geq 0.$ Here, the initial positions are i.i.d. such that for any $i,$  $
 \bar X_0^{i}\sim \nu_0.$ Moreover, $(\bar\pi^{i})_i$ is a family of i.i.d. Poisson random measures on $\R^2_+$ having intensity $ ds dz,$ independent of the collection of initial conditions $(\bar X^i_0)_{i \geq 1}, $  and $S^\alpha $ is a strictly $\alpha$-stable process independent of  the collection of Poisson random measures $(\bar\pi^{i})_i$ and of the $(\bar X^i_0)_{i \geq 1}. $ 
  The measure $\bar \mu$ on $(D (\R_+, \R) ,\D(\R_+, \R))$ is the regular conditional distribution of $\bar X^i$ given $S^\alpha, $  and $\bar \mu_s$ is its projection onto time $s$.
Lemma \ref{lem:mut} in the Appendix section shows that the stochastic process $\bar \mu_t(f)=\E[f(\bar X^i_t)|S^\alpha]=\int_{ \R }f(\gamma)d\bar\mu_t(\gamma)$ is c\`adl\`ag, and hence the integral with respect to the stable process is well defined. 
Finally, we recall that  $\psi \equiv 0$ if $ \alpha > 1.$ 

It was shown in \cite{LLM} that under Assumptions \ref{ass:b}-\ref{ass:psi} and \ref{ass:nu0}, the system \eqref {eq:limit:system} admits a unique strong solution. Moreover, the infinite system \eqref {eq:limit:system} is exchangeable. Since the Skorokhod space is a Borel space, using the De Finetti-Hewitt-Savage theorem (see for instance \cite{aldous}), the law of \eqref {eq:limit:system} is the law of a mixture of i.i.d. $D(\R_+, \R)-$ valued random variables which is directed by some random measure. It is easy to see that  $\mathcal{L}(\bar X^1 | S^\alpha)$ is in fact the directing measure of  \eqref{eq:finite:system} (see  Proposition \ref{thm:dir_m} below for a proof).

\subsection{Main results} 
Recall that we denoted by $\mu^N$  the empirical measure of \eqref{eq:finite:system} and by $(\bar X^{i})_{i \geq 1 }$ the unique strong solution of \eqref {eq:limit:system}. Recall also that $ \psi \equiv 0 $ if $ \alpha > 1$ in  \eqref{eq:finite:system} and \eqref {eq:limit:system}.
\begin{theorem}\label{theo:main}
Grant Assumptions \ref{ass:b}-\ref{ass:weak_cont}.  Then  $\mu ^N$  converges in distribution, in ${\mathcal P} ( D ( \R_+ , \R ) ) ,$ to $\bar \mu   = {\mathcal L}( \bar X^1 | S^\alpha ) . $
\end{theorem} 

Assume that $D(\R_+ ,\R)^{\N^*}$ is endowed with the product
topology. Then the following result is a consequence of Proposition 7.20 of \cite{aldous}. Assume that the sequence $(X^{N, i })_{ 1 \le i \le N} $ is completed by $X^{N, i }=0$ for $i>N.$ 

\begin{corollary}\label{cor:main}
Grant the assumptions of Theorem \ref{theo:main}. Then $(X^{N, i })_{ 1 \le i \le N} $  converges in distribution, in $D(\R_+ ,\R)^{\N^*},$ to $ (\bar X^i )_{ i \geq 1 }.$  \end{corollary}
The proof of Theorem \ref{theo:main} is given in Subsection \ref{sec:proof_theo_main}. The main arguments that allow to prove Theorem \ref {theo:main}  are collected in the following  subsection.

\subsection{Main steps of the proof of Theorem \ref{theo:main}}
 
{\bf Step 1:} {\bf Tightness of the sequence $\pmb{\mu^N}$.} 
As it is classically done, we start by proving that $(\mu^N)_N$ is tight in $\cP(\R_+, \R)$.

\begin{proposition}\label{thm:tightness}
Grant all assumptions of Theorem \ref{theo:main}.  Then the sequence of empirical measures $(\mu^N)_N$ of \eqref{eq:finite:system} is tight. 
\end{proposition}

The proof of Proposition \ref{thm:tightness} is given in Section \ref{sec:tightnes_muN}. We give below the idea of this proof. 

Contrarily to the situation where $\nu$ possesses  moments of order two, 
proving tightness here is one of the subtle and original steps of our proof. The tightness of $(\mu^N)_N$ is equivalent to that of $(X_t^{N,1})_N$, which, in the case where $\nu$ possesses moments of second order, is shown using Aldous' tightness criterion. Here, because of the normalization $N^{-1/\alpha}$ and the fact that $\nu$ possesses only moments of order less than $\alpha,$  it does not seem possible to rely on that criterion. To prove tightness, the main observation is that the sum of the collateral jumps 
\begin{equation}\label{eq:rw_collj}
J_t^N \coloneqq \frac{1}{{N}^{ 1/\alpha} }\sum_{ j =1}^N \int_{[0,t]\times\R_+\times\R  } u \indiq_{ \{ z \le  f ( X^{N, j}_{s-}) \}} \pi^j (ds,dz,du), t \geq 0, 
\end{equation}
is a compound Poisson process observed after a random time change, where the time change is given by the integrated jump intensity of the whole system. Since the jump heights of the compound Poisson process (the $u-$variables in the above integral) belong, by assumption, to the domain of attraction of a stable law, the stable central limit theorem then allows to show that the sequence $ (J^N)_N$ is tight
in $ D (\R_+ , \R ) , $ and 
the tightness of the sequence $ (\mu^N)_N$ follows easily from this.  
\\

\noindent{\bf Step 2: }{\bf  Representation result.} The next important step is the following representation result for the limit system obtained along a given converging subsequence. 
\begin{theorem}\label{theo:main-repr} 	
Grant the assumptions of Theorem \ref{theo:main} and let $\mu$ be the limit in law along a fixed convergent subsequence of the sequence $(\mu^N)_N$ of empirical measures of system \eqref{eq:finite:system}. Moreover, let $(X^{i})_{i\geq 1}$ be an infinite-exchangeable system directed by $\mu$. Then there exist, on an extension of the probability space where $ \mu $ and $(X^{i})_{i\geq 1}$ are defined, a strictly $\alpha$-stable process $ S^\alpha, $ independent of the $ (X^i_0)_{ i \geq 1 }, $ and a family of  i.i.d. Poisson random measures $(\bar\pi^{i})_{i\geq 1}$ with Lebesgue intensity, independent of $S^\alpha $ and of  the $ (X^i_0)_{ i \geq 1 },$ such that 
 $(X^{i})_{i\geq 1}$ admits the representation 
		\begin{equation}\label{eq:limit:system_subseq}
    		\begin{split}
    			X^i_t & = X^i_0 + \int_0^t b(X^i_s, \mu_s) ds + \int_{[0,t]\times \R_+} \psi(X^{i}_{s^-}, \mu_{s^-}) \indiq_{\{ z \leq f(X^{i}_{s^-})			\}} \bar\pi^i(ds,dz)  \\
    			& + \int_{[0,t]} \mu^{1/\alpha}_{s^-}(f) d S^\alpha_s , \qquad i\geq 1 ,
    		\end{split}
    	\end{equation}
	for all $t>0$. 
\end{theorem}	
Theorem \ref{theo:main-repr} is proven in Section \ref{sec:passage_to_lim:alpha<1} for $\alpha<1$ and in Section \ref{sec:passage_to_lim:alpha>1} for $\alpha>1.$ \\

 \noindent{\bf Step 3:} {\bf Identification result}.
 Notice that with respect to \eqref{eq:limit:system}, in the representation \eqref{eq:limit:system_subseq}, the main information that is still lacking is the precise form of the directing measure which in \eqref{eq:limit:system} is given by $ \bar \mu = {\mathcal L} (\bar X^1 | S^\alpha ).$ We complete this point in the following and last step. 
  \begin{theorem}\label{theo:main-ident}
	Grant the assumptions of Theorem \ref{theo:main}. Let $\mu$ be the limit in law of the sequence $(\mu^N)_N$ along a
fixed convergent subsequence  and $(X^{i})_{i\geq 1}$ be an 
infinite-exchangeable system directed by $\mu$, with its
	representation given in \eqref{eq:limit:system_subseq}. Let  $ ( \Omega, {\mathcal A}, P ) $ be the probability space on  which the initial conditions  $ (X^i_0)_{ i \geq 1 }, $ the strictly $\alpha$-stable process $ S^\alpha, $ 
	and the family of 
	Poisson random measures $(\bar\pi^{i})_{i\geq 1},$ 
	appearing in \eqref{eq:limit:system_subseq} are defined. 	    		
Let $ (\bar X^i)_{  i \geq 1 } $ be the strong solution of \eqref{eq:limit:system}, defined on $ ( \Omega, {\mathcal A}, P ), $  driven by $ S^\alpha $ and  $ (\bar \pi^i )_{i\geq 1}$ and starting from $(X^i_0)_{ i \geq 1 }$ .  
Then for all $i\geq 1$ and $t \geq 0,$ 
$ \bar X^i_t = X^i_t  \mbox{ a.s.}$ and $ \mu = \bar \mu \; a.s. $
\end{theorem} 
Theorem \ref{theo:main-ident} is proven in Section \ref{sec:uniqueness}.

\subsection{Proof of Theorem \ref{theo:main}}\label{sec:proof_theo_main}
\begin{proof}
Theorems \ref{theo:main-repr} and \ref{theo:main-ident} together yield that every convergent subsequence of the sequence $(\mu^N)_N$ of the empirical measures of system \eqref{eq:finite:system} converges in law to the same limit $\bar\mu = \law{\bar X^1 | S^\alpha}$. 
Hence, the whole sequence $(\mu^N)_N$ converges to such a limit, and this concludes the proof of Theorem \ref{theo:main}. 
\end{proof}

\section{Tightness of $(\mu^N)_N$}\label{sec:tightnes_muN}
In this section, we will 
show tightness of the sequence of empirical measures $ (\mu^N)_N,$ i.e.
  tightness of the sequence of their laws on $\mathcal{P}(D ( \R_+, \R)) $. This is the result of   Proposition \ref{thm:tightness}, which will be proven at the end of the section. The main ingredient to prove this is the following result. 
 Recall the definition of $ J_t^N$ in \eqref{eq:rw_collj} above.
\begin{proposition}\label{prop:tightness_J}
Grant all assumptions of Theorem \ref{theo:main}.  Then the sequence $((J^N_t)_{ t \geq 0 })_N$ is tight
 in $ D ( \R_+, \R ).$ 

\end{proposition}

\begin{proof}
We define the cumulated jump intensity -- up to a factor $N$ -- of system \eqref{eq:finite:system} by
\begin{equation*}
	A^N_t \coloneqq  \int_0^t  \mu^N_{s}(f)  ds .
\end{equation*}
Notice that $A^N_0 =0$, $t \mapsto A^N_t$ is strictly increasing (since $f$ is strictly positive) and continuous, hence predictable, that $A^N_\infty = \infty$ and that $\underline{f} (t-s) \leq A^N_t - A^N_s \leq \norme{f}_\infty (t-s)$ for each $0 \le s<t < \infty . $ So $A^N$ is a random variable with values in $\Lambda^{ \underline{f}, \norme{f}_\infty }.$ 
Define now the process $(\tau^N_t)_{t\geq 0}$ by 
\[
	\int_{0}^{\tau^N_t} N \mu^N_{s}(f)   ds =  t , 
\]
that is, 
\[
	N A^N_{\tau^N_t} = t.
\]  

\noindent{\bf Step 1.} Here we show that 
\[
(\tilde J_t)_t \coloneqq (N^{ 1/ \alpha}  J^N_{\tau^N_t})_t
\]
is a compound Poisson process with unit rate and jump distribution $ \nu.$ 

To prove this fact, let $ \bar N ( dt, du) $ be the jump measure associated to the process $ N^{ 1/ \alpha} J^N, $ that is,
$$ \bar N ( A \times B ) = \sum_{ t \geq 0} \indiq_{\{ t \in A \} } \indiq_{ \{ N^{ 1/ \alpha}\Delta J_t^N \in B \} } , $$
for any $ A \subset \R_+ $ and $ B \subset \R \setminus \{ 0\} .$ 

$ \bar N $ is a marked point process in the sense of Definition 5.1.22 of \cite{bremaud}. Its intensity measure is given by 
$$  \sum_{ i = 1 }^N f( X_{t}^{N, i } ) dt \nu ( du) = N \mu^N_t (f) dt \nu (du )  .$$
In particular, in the sense of Definition 5.1.23 of \cite{bremaud}, the intensity kernel $ \lambda (t, du ) $ associated to $ \bar N $ is given by 
$$ \lambda ( t, du) = N \mu^N_t(f)   \nu( du), $$ 
such that, with the notation of page 174 of \cite{bremaud}, $ \Phi( t, du ) = \nu (du) $ does not depend on $t. $ 

Let now $ \tilde N (dr, dz) $ be the jump measure associated to the time-changed process $ \tilde J_t^N = N^{ 1/ \alpha}J^N_{\tau^N_t} .$ We check that its intensity measure is given by $ dr \nu (dz) . $ Indeed, to see this, let $0 \le  a < b $ and $ B \subset  \R\setminus \{ 0\} .$ Notice that  $ \tau_a^N < \tau_b^N $ are both stopping times. Let moreover $ A \in \mathcal F_{\tau^N_a} . $ The process $ (t, u)  \mapsto \indiq_A \indiq_{\{ \tau_a < t \le \tau_b \}} \indiq_{\{ u \in B \}} $ is predictable, so that, using the explicit form of the intensity of $ \bar N, $ we have
\begin{multline*} 
\E [\tilde N ( ]a, b ] \times B )  \indiq_A ] = \E \left[ \indiq_A \int_{ ]\tau^N_a, \tau^N_b ] } \int_{\R} \indiq_B (u ) \bar N ( dt, du ) \right] \\
= \E \left[ \indiq_A \int_{ ]\tau^N_a, \tau^N_b ] } \int_{\R} \indiq_B (u )  N \mu^N_t (f)  dt \nu (du )  \right] \\
= \nu (B)  \E \left[ \indiq_A \int_{ ]\tau^N_a, \tau^N_b ] }  N \mu^N_t ( f)  dt \right] \\
= \nu( B) \E [ \indiq_A N [ A^N_{\tau_b^N} - A^N_{\tau_a^N} ] ] = \nu (B) (b-a)  \P(A) ,
\end{multline*}
which concludes the proof.

As a consequence of the above result, $\tilde N( dr, dz ) = \sum_{n } \delta_{(\tilde S_n, U_n) } $ is a Poisson random measure, with intensity $ dr \nu ( dz) $ of product form, which implies the independence of the sequences $ (\tilde S_n)_n$ and $ (U_n)_n.$ In particular, 
$$ \tilde J_t^N = \int_{[0,t ] } \int_\R z \tilde N (dr, dz ) = \sum_n  U_n \indiq_{ \{ \tilde S_n \le t \} }  = \sum_{n = 1 }^{ \tilde P_t} U_n, $$
if we put $\tilde P_t =  \sum_n   \indiq_{ \{ \tilde S_n \le t \} } ,$ a Poisson process of rate $1,$ independent of $ (U_n)_n.$ This concludes the first step. \\

\noindent{\bf Step 2.} Using the above notation, let us introduce the rescaled random walk
\[
	S^N_k \coloneqq \frac{1}{N^{1/\alpha}} \sum_{l=1}^{k} U_l , \qquad k \geq 0 .
\]
As a consequence of Step 1, we have the representation
\begin{equation} \label{eq:J^N_composition} 
J^N 	= S^N \circ \tilde{P}_{N\cdot} \circ A^N  , 
\end{equation}
that is, $J^{N}_t = S^N_{\tilde P_{N A^N_t}} $ for all $ t \geq 0, $ 
where the unit rate Poisson process $\tilde P$ is independent of the random walk $ S^N.$ 
Since $ \nu$ belongs to the domain of attraction of a strictly stable law, with renormalization sequence $ b_n = n^{ 1 / \alpha}, $ Corollary 16.17 in \cite{kallenberg} implies convergence in law of 
\begin{equation}\label{eq:process_converging_to_Salpha}
   ( S^N_{\tilde P_{N t }})_t  \to  (S_t^\alpha )_t ,
   \end{equation} 
with respect to the topology of locally uniform convergence. In particular, the sequence $(S^N_{\tilde P_{N \cdot}})_N = ( S^N \circ \tilde{P}_{N\cdot} )_N $ is tight in $D ( \R_+, \R).$  \\

\noindent{\bf Step 3.}  
We conclude our proof by showing that the above implies that the sequence  $(J^{N}_\cdot )_N = ( S^N \circ \tilde{P}_{N\cdot} \circ A^N_\cdot)_N$ is tight. Indeed, the sequence $(S^N_{\tilde P_{N \cdot}} , A^N_\cdot)_N$ of processes with values in $D (\R_+, \R ) \times \Lambda^{ \underline{f}, \norme{f}_\infty }$ is tight since its components are:  for $(A^N )_N$, this follows from Theorem 7.3 of \cite{billingsley}. Moreover, the composition map $\Phi : D ( \R_+, \R )  \times\Lambda^{ \underline{f}, \norme{f}_\infty } \to  D ( \R_+, \R ), \, (\gamma, \eta) \mapsto \gamma \circ \eta$ is continuous (see Lemma \ref{lem:continuity_of_composition}). Since 
tightness
 is preserved by continuous mappings, $(J^{N}_\cdot )_N$ is tight . 
\end{proof}

 Using the previous result, we can now give the proof of  Proposition \ref{thm:tightness}.  
\begin{proof}[Proof of Proposition \ref{thm:tightness}]
Thanks to Proposition 2.2 in \cite{sznitman}, to prove Proposition \ref{thm:tightness} we can equivalently show that, for any fixed $i \in \{1,\ldots, N\}$, the sequence of processes $(X^{N,i})_N = ( (X^{N,i}_t)_t)_N$ is tight. Let us write
\begin{equation}\label{eq:decxni}
	X^{N,i}_t = X^{N,i}_0 + B^{N,i}_t + I^{N,i}_t + J^{N}_t  - E^{N,i}_t ,
\end{equation}
where
\[
	B^{N,i}_t \coloneqq  \int_0^t b(X^{N,i}_s, \mu^N_s) ds , 
\]
\begin{equation*}
	E^{N,i}_t \coloneqq \frac{1}{N^{1/\alpha}}\int_{[0,t]\times \R_+\times \R} u \indiq_{\{ z \leq f(X^{N,i}_{s^-})\}} \pi^{i}(ds,dz,du) 
\end{equation*}
and 
\begin{equation*}
	I^{N,i}_t \coloneqq  \int_{[0,t] \times \R_+ \times \R} \psi(X^{N,i}_{s^-}, \mu^N_{s^-}) \indiq_{\{ z \leq f(X^{N,i}_{s^-})\} }  \pi^{i}(ds,dz,du) ,
\end{equation*}
with $\psi \equiv 0$ in the case $\alpha>1. $ Recall that $J^{N}_t $ was defined in \eqref{eq:rw_collj}.

Therefore, to prove tightness of the sequence $(X^{N,i})_N$, it suffices to prove tightness of the sequences $(B^{N,i})_N$, $(I^{N,i})_N$, $(J^{N})_N$ and $(E^{N,i})_N$  separately. 
By the boundedness assumptions on $b$, $f$ and $\psi$, $(B^{N,i})_N$ and $(I^{N,i})_N$ are easily shown to be tight. Tightness of $(J^N)_N$ was proved in Proposition \ref{prop:tightness_J}. So we only need to show that  $(E^{N,i})_N$ is tight.

Using Aldous' criterion (see  Theorem VI.4.5 in \cite{JS}), this reduces to showing that: 
\begin{enumerate}
	\item for all $ M\in \N^\ast, \, \epsilon>0$ there exist $N_0\in\N^\ast,\, K > 0 $ such that
	\[ 
		N \geq N_0 \implies \P \left(\sup_{t\leq M} |E^{N,i}_t| > K  \right) \leq \epsilon ;
	\]
	\item for all $ M\in \N^\ast, \, \epsilon>0$, 
	\[
		\lim_{\theta \downarrow 0} \limsup_{N} \sup_{ \{ S,T  \, : \, S \leq T \leq S + \theta \} } \P \left( | E^{N,i}_T - E^{N,i}_S | \geq \epsilon \right) =  0 ,
	\]
	where the supremum is taken over all stopping times that are bounded by $M.$ 
	\end{enumerate}
In what follows, we fix some $\am<\alpha$ in the case  $\alpha<1,$ and we take $ \am = 1 $ if $ \alpha > 1.$ Using that the integral with respect to $\pi^{i}$ defining $ E^{N, i } $ is a.s. a sum with a finite number of terms (by boundedness of $f$), by sub-additivity,

\begin{align*}
	 |E^{N,i}_t |^{\am } & \leq 
	\frac{1}{N^{{\am }/\alpha}} \int_{[0,t]\times \R_+ \times \R} |u|^{\am } \indiq_{\{ z \leq f(X^{N,i}_{s^-})\}} \pi^{i}(ds,dz,du) \\
	& \leq \frac{1}{N^{{\am }/\alpha}}\int_{[0,t]\times \R_+ \times \R} |u|^{\am }  \indiq_{\{ z \leq \norme{f}_\infty\}} \pi^{i}(ds,dz,du).
\end{align*}
Thus, by Markov's inequality,
\begin{align*}
	\P\left(\sup_{t\leq M} |E^{N,i}_t | \geq K\right) & = \P \left( \sup_{t\leq M} |E^{N,i}_t |^{\am }\geq K^{\am } \right) \\
	& \leq 
	 \frac{1}{N^{\am /\alpha} K^{\am}} \E \left[\sup_{t\leq M}\int_{[0,t]\times \R_+ \times \R^\ast} |u|^{\am } \indiq_{\{ z \leq \norme{f}_\infty\}} \pi^{i}(ds,dz,du) \right] \\
	 & \leq  \frac{1}{N^{\am /\alpha} K^{\am }} \E \left[\int_{[0,M]\times \R_+ \times \R^\ast} |u|^{\am} \indiq_{\{ z \leq \norme{f}_\infty\}} ds dz \nu(du) \right] \\
	 & \leq C \frac{M}{N^{\am /\alpha} K^{\am }} ,
\end{align*}
for some positive constant $C$, since $\int_{\R^\ast} |u|^{\am }\nu(du) < +\infty$, so the first point is proved for any $\alpha\in (0,2)\setminus \{1\}$. 

Then, take two stopping times such that  $S\leq T \leq S+\theta \le M,$ for some $ \theta > 0.$ We obtain similarly the upper bound   
\begin{equation*}
	\P\left( |E^{N,i}_T - E^{N,i}_S  | \geq \epsilon \right) \leq C \frac{\E ( | T - S|)}{N^{\am/\alpha} \epsilon^{\am }} , 
\end{equation*}
such that
\begin{equation*}
	 \sup_{ \{ S,T  \, : \, S \leq T \leq S + \theta \le M  \} } \P \left( |E^{N,i}_T - E^{N,i}_S  | \geq \epsilon \right) \leq C \frac{\theta}{N^{\am /\alpha} \epsilon^{\am }} ,
\end{equation*}
and the second point is also checked to hold.
\end{proof}

\section{Proof of Theorem \ref{theo:main-repr}}\label{sec:5}
Thanks to Proposition \ref{thm:tightness}, $(\mu^N)_N$ admits at least a convergent subsequence. Let us take one such converging subsequence, 
still denoted (for simplicity) by $(\mu^N)_N,$ and let  $\mu$ be its limit in law. $ \mu $ is thus a random element of $\mathcal{P}(D( \R_+, \R))$. To prove Theorem \ref{theo:main-repr}, we need to characterize further this limit $ \mu, $ and this is done in this section.

\subsection{A common convergence result}\label{sec:common_convergence_results}
In this section we state a first consequence of the convergence of $ \mu^N $ to $ \mu$  that will be needed in both cases $\alpha<1$ and $\alpha>1$, namely, the joint convergence of  $((X^{N,i})_{ i= 1}^m, \mu^N)_N$, for any $ m \geq 1,$ to its corresponding limit quantity.   

More precisely, consider the infinite-exchangeable system $ X = (X^i)_{i\geq 1}$ having $\mu$ as its directing measure. 
By Proposition (7.20) of \cite{aldous}, the convergence in law  $(\mu^N)_N \convlaw\mu$ implies convergence in law $(X^{N,i})_{i=1}^N \convlaw (X^i)_{i\geq 1}$, that is, for any $m\geq 1$, 
\[	
	(X^{N,1}, \dots, X^{N, m}) \convlaw (X^{1},\dots, X^{m}) 
\]
(see Eq. (7.19) in \cite{aldous}). 

The following result follows then immediately by exchangeability, see also Proposition (7.20) of \cite{aldous}. 

\begin{proposition}\label{prop:convergence_couple}
Let $\mu$ the limit in distribution of a subsequence of $(\mu^N)_N$, still denoted by $(\mu^N)_N$. Let $X = (X^i)_i$ be the infinite-exchangeable system directed by $\mu$. Then, for all $m\geq 1$, for this subsequence, 
\[
	( (X^{N,i})_{i=1}^m, \mu^N) \convlaw ( (X^i)_{i=1}^m, \mu) .
\] 
\end{proposition}

\begin{proof}
	We show that for any continuous bounded test function $G: D(\R_+, \R)^m \times \mathcal{P}(D(\R_+, \R))$, $\lim_N \E[G( (X^{N,i})_{i=1}^m, \mu^N )] = \E[ G( (X^ i)_{i=1}^m, \mu ) ]$.
	By density arguments, we can take $G$ of the form $G( \gamma_1, \ldots, \gamma_m, p) = g_1(\gamma_1) \cdot \ldots \cdot g_m(\gamma_m) g(p)$, with $g$ continuous and bounded on $\mathcal{P}(D(\R_+, \R))$ and $g_1, \ldots, g_m$ continuous and bounded on $D(\R_+, \R)$.  
	
	Suppose $m=1.$ Then by exchangeability,  
	$$ \E [ g_1 ( X^{N, 1} ) g ( \mu^N) ] = \frac1N \sum_{j=1}^N \E [ g_1 ( X^{N, j} ) g ( \mu^N) ] = \E \left[ g ( \mu^N ) \int_{D(\R_+, \R)} g_1 ( \gamma )  \mu^N ( d \gamma )\right] .$$
	Since $ \mu \mapsto g ( \mu )  \int_{D(\R_+, \R)} g_1 ( \gamma )  \mu ( d \gamma )$ is continuous and bounded, the above expression converges 
	to $ \E [ g ( \mu ) \int_{D(\R_+, \R)} g_1 ( \gamma )  \mu ( d \gamma )] = \E [ g_1 ( X^{ 1} ) g ( \mu) ],$ as $ N \to \infty.$ 
	 
	This argument extends to the joint law of $m$ trajectories: for all $N\in \N$, $1\leq m \leq N$, up to an error $ r_N (m) $ converging to $ 0, $ as $N \to \infty , $ 
	\begin{multline*}
	 \E [ g_1 ( X^{N, 1} ) \ldots g_m ( X^{N, m } ) g ( \mu^N) ] = \\
	 \E \left[ g ( \mu^N ) \int_{D(\R_+, \R)^m } g_1 ( \gamma_1 ) \ldots g_m ( \gamma_m)  (\mu^N )^{\otimes m } ( d \gamma_1, \ldots , d \gamma_m)\right]  + r_N (m), \end{multline*}
	such that 
	\begin{multline*}
 		\lim_{N \to \infty } \left( \E [ g_1 ( X^{N, 1} ) \ldots g_m ( X^{N, m } ) g ( \mu^N) ] \right. -\\
  		\left. \E \left[ g ( \mu^N ) \int_{D(\R_+, \R)^m }  g_1 ( \gamma_1 ) \ldots g_m ( \gamma_m) (\mu^N )^{\otimes m } ( d \gamma_1, \ldots , d \gamma_m) \right] \right) = 0 .
	\end{multline*}   
	As in the case $m=1,$ the last expression is a continuous and bounded functional of $ \mu^N, $ such that the weak convergence of $\mu^N$ to $\mu$ and the fact that the limit system is directed by $\mu$ imply that 
	\begin{multline*}
		\lim_{N \to \infty }  \E \left[ g ( \mu^N ) \int_{D(\R_+, \R)^m }  g_1 ( \gamma_1 ) \ldots g_m ( \gamma_m) (\mu^N )^{\otimes m } ( d \gamma_1, \ldots , d \gamma_m)\right]  \\= 
  		\E \left[ g ( \mu ) \int_{D(\R_+, \R)^m } g_1 ( \gamma_1 ) \ldots g_m ( \gamma_m) \mu^{\otimes m } ( d \gamma_1, \ldots , d \gamma_m)  \right]\\
  		= \E [ g_1 ( X^{ 1} ) \ldots g_m ( X^{m } ) g ( \mu) ].
	\end{multline*}
\end{proof}

\subsection{Convergence of $J^N$}
 The following result takes a particular place in the paper. Though it is not used in the proofs of the main theorems, it expresses the main idea behind the limiting behavior of the simultaneous jump term. In the case where $b=0$ and $\psi=0$, this proposition would be sufficient to prove the convergence of the finite system to its limit. 

\begin{proposition}\label{prop:convergence_J}
	Grant the assumptions of Theorem \ref{theo:main} and let $J^N$ be as in \eqref{eq:rw_collj}. Then $ ( J^N, \mu^N)_N$ is tight, and any converging subsequence possesses a weak limit of the form  $ ( J^\infty , \mu ) ,$ 
	where          
        \begin{equation*}
        	J^\infty  = S^\alpha\circ A ,
        \end{equation*}
        with $ S^\alpha$ a strictly $\alpha$-stable process and  
	\begin{equation*}
	    A_t \coloneqq \int_0^t \mu_s(f) ds , t \geq 0.
        \end{equation*}
        \end{proposition}

\begin{proof}
Tightness of $( J^N, \mu^N)_N$ follows from the fact that each coordinate is tight. In what follows, we therefore consider any converging subsequence. In the proof of Proposition \ref{prop:tightness_J} we showed that $J^{N} =  S^N \circ \tilde{P}_{N\cdot} \circ A^N_\cdot $ (see \eqref{eq:J^N_composition}) and that $S^N \circ \tilde P_N \convlaw S^\alpha$ (see \eqref{eq:process_converging_to_Salpha}). Lemma \ref{lem:continuity_of_composition} of the Appendix section shows that the composition map $\Phi : D ( \R_+ , \R)  \times \Lambda^{\underline f , \norme{f}_\infty}  \to D ( \R_+, \R ) , \, (\gamma, \eta) \mapsto \gamma \circ \eta$ is continuous, so it preserves convergence in law. Therefore, if we show that $A^N \convlaw A$, then we have that $J^{N} \convlaw S^\alpha \circ A$.  

Lemma \ref{lem:app2} of the Appendix section shows that the map 
\[
	\mathcal{P}(D(\R_+, \R)) \to \Lambda^{\underline f , \norme{f}_\infty}  , \, m \mapsto \int_0^\cdot m_s(f) ds 
\]
is continuous at any point $ \tilde m $ satisfying that for any fixed $s,$  $\tilde m (\{ \gamma \in D(\R_+, \R) \, | \, \gamma(s) \neq \gamma(s^-)\})=0.$ 

So it suffices to show that for any $s$, $\mu(\{ \gamma \in D(\R_+, \R) \, | \, \gamma(s) \neq \gamma(s^-)\})=0$ a.s., which is proved in Lemma \ref{lem:supp_mu}, also stated in the Appendix section. This concludes this proof. 
\end{proof}

The above proposition is the main ingredient of our proof. To identify any possible limit,  we need to determine the dependency structure of $ S^\alpha $ and $ \mu ,$ that is, the joint law of $ ( S^\alpha , \mu ) ,$  and this is done in the remaining part of this section, first in case $ \alpha < 1, $ and then in case $ \alpha > 1.$ 
\subsection{Proof of Theorem \ref{theo:main-repr} for  $ \alpha < 1 $}\label{sec:passage_to_lim:alpha<1}
Let us consider the two-dimensional semimartingale $X^{N,i,k} \coloneqq (X^{N,i},X^{N,k}).$ 
The triplet $(X^{N, i }, X^{ N, k}, \mu^N)_N$ is tight, since its components are (Proposition \ref{thm:tightness} and its proof).  
We consider a fixed convergent subsequence that converges in law to a limit $(X^{ i }, X^{  k}, \mu ) $ where, by Proposition  \ref{prop:convergence_couple}, the joint law of  $(X^{  i }, X^{  k}) $ is the law of two coordinates of an infinite-exchangeable system directed by $ \mu$. In this section we will identify the structure of the limit $ X^{i, k }=(X^{  i }, X^{  k}). $  
The method we use relies on standard results from semimartingale theory, as presented for instance in \cite{JS}, but it is less common in the setting of interacting particle systems.  
The main idea is the following: once the characteristics of the semimartingales $X^{N,i,k}$ are shown to converge in distribution to some limiting characteristics, Theorem IX.2.4 of \cite{JS} allows us to conclude that $X^{N,i,k}$ converges to a semimartingale $X^{i,k} = (X^{i}, X^{k})$  with these limiting characteristics. This corresponds to Step 1 below. 
Step 2 consists in identifying this limiting semimartingale, whose characteristics are now known, as a process satisfying a stochastic differential equation of the same form as our limiting system, with coefficients depending on $\mu$. This part relies on a representation theorem, Theorem II.7.4 of \cite{IW}. 
At this stage, however, we do not yet know that $\mu$ is the conditional law, given the driving process $S^\alpha$, of one of the coordinates in this representation. Therefore, the identification of the limiting system is not yet complete. This will be established in Theorem~\ref{theo:main-ident}. 
The case $\alpha>1$, treated in the next section, follows the same strategy.

%
%
\begin{proof}

\noindent 

\noindent{\bf Step 1.} 

Since for $\alpha<1$ all the integrals with respect to the Poisson measures appearing in the definition of $X^{N, i }$  are of locally bounded variation, we can choose the truncation function  $h \equiv 0$ in Definition II.2.3 of \cite{JS}. With the notation of equation II.2.4 in \cite{JS}, we set, for this choice of $h, $ 
  \begin{equation*}
X^{N, l}_t(h)= X^{N, l }_0+ B^{N,l}_t(h)+M^{N,l}_t(h), \qquad l=i,k ,
\end{equation*}
where  $M^{N,l}(h)=0.$ This implies (see Definition II.2.6 in \cite{JS}) that the characteristics of $X^{N,i,k}$ are $(B^{N,i,k}, C^{N,i,k}, \nu^{N,i,k})$ where $B^{N,i,k} = (B^{N,i}, B^{N,k})$, and 
\begin{align*}
&B^{N,l}_t =\int_0^t b(X^{N,l}_s,\mu^{N}_s)ds , \qquad l=i,k ,\\
&C^{N,i,k}=0,
\end{align*}
and the modified second characteristic (see Definition II.2.16 in \cite{JS}) $\tilde C^{N,i,k}=0.$ The compensator of the jump measure of $X^{N,i,k}$ is

\begin{equation*}
\nu^{N,i, k}(ds, dx_1, dx_2) = \sum_{l=1}^4\nu^{N,i,k}_l ( ds, dx_1, dx_2) , 
\end{equation*}
with 
\begin{align*}
	 \nu^{N,i,k}_1 (ds, dx_1, d x_2) & =  f(X^{N,i}_s) ds \int_{\R^\ast} \nu(du) \delta_{( \psi(X^{N,i}_s, \mu^N_s), u/N^{1/\alpha})}(dx_1, dx_2 ), \\
	\nu^{N,i,k}_2 (ds, dx_1, dx_2)& =  f(X^{N,k}_s) ds \int_{\R^\ast} \nu(du) \delta_{( u/N^{1/\alpha},\psi(X^{N,k}_s, \mu^N_s) )}(dx_1, dx_2 ),\\
	\nu^{N,i,k}_3 (ds, dx_1, dx_2)&=  \mu^N_s ( f) ds  \int_{\R^\ast} \nu(du) N \delta_{(  u/N^{1/\alpha}, u/N^{1/\alpha})}(dx_1, dx_2) , \\
	\nu^{N,i,k}_4 (ds, dx_1, dx_2)&  =-  \left(  f(X^{N,i}_s) +  f(X^{N,k}_s)  \right) ds \\
	& \quad \quad \quad \quad \quad \quad \quad \quad \quad \quad \int_{\R^\ast} \nu(du) \delta_{( u/N^{1/\alpha}, u/N^{1/\alpha})}(dx_1, dx_2 ) .
	\end{align*}

We now introduce the corresponding limit objects which are $B^{i,k} =(B^{i},B^{k}) ,$ the $2\times 2$ symmetric non-negative matrix  $\tilde C^{i,k}$ and the measure $\nu^{\infty, i,k}(ds,dx)$ given by 
\begin{align*}
& B^l_t=\int_0^tb(X^{l}_s,\mu_s)ds , \qquad l=i,k , \\
& \tilde C^{i,k}=0 ,
\end{align*}
\begin{equation}\label{eq:def_nu^infty,i}
\nu^{\infty, i, k}(ds,dx) =\nu^{\infty, i, k}_1(ds,dx) + \nu^{\infty, i, k}_2(ds,dx) + \nu^{\infty, i, k}_3(ds,dx) , \qquad x= ( x_1, x_2 ), 
\end{equation}
where
\begin{align*}
\nu^{\infty, i, k}_1(ds,dx_1, dx_2 ) & \coloneqq  f(X^{i}_s) ds \delta_{(\psi(X^{i}_s, \mu_s), 0)}(dx_1, dx_2) , \nonumber  \\ 
	\nu^{\infty, i, k}_2(ds,dx_1, dx_2) & \coloneqq  f(X^{k}_s) ds \delta_{(0, \psi(X^{k}_s, \mu_s))}(dx_1, dx_2) , \nonumber\\
	\nu^{\infty, i, k}_3(ds,dx_1, dx_2 ) & \coloneqq \int_{\R^\ast} \nu^\alpha(du) \delta_{(\mu^{1/\alpha}_s(f) u, \mu^{1/\alpha}_s(f) u)}(dx_1, dx_2) ds , \nonumber
\end{align*}
and where $\nu^{\alpha}$ is the L{\'e}vy measure given by \eqref{eq:nualpha}. The limit measure corresponding to $\nu^{N,i,k}_4$ equals zero. 

To employ Theorem IX.2.4 in \cite{JS}, we need to check that, as $ N \to \infty, $ 
\begin{equation}\label{eq:cvb}
(X^{N,i,k}, B^{N,i,k}, \tilde C^{N,i,k} )\convlaw (X^{i,k} ,B^{i,k},\tilde C^{i,k})
\end{equation} 
in $D(\R_+, \R^{2+2+ 4}),$ 
and that
\begin{equation}\label{eq:cvnu} 
 \left(X^{N,i,k}, \int_{[0,\cdot]\times \R^2} g d\nu^{N,i,k} \right) \convlaw \left(X^{i,k} , \int_{[0,\cdot]\times \R^2} g  d\nu^{\infty, i,k} \right)
 \end{equation}
in $D(\R_+, \R^{2+1})$, for all functions $g\in \mathcal{C}_1 (\R^2 ) $ (recall the Definition VII 2.7 of \cite{JS}), which are continuous, bounded functions $g: \R^2 \to \R$ which are zero around zero. 
    
Using Lemma \ref{lem:app5}, $B^{N,l}$ is obtained by applying a continuous mapping to the originally converging subsequence $(X^{N,l},\mu^N),$ hence \eqref{eq:cvb} follows.

Now we establish the convergence in law \eqref{eq:cvnu}. 
Denote 
$$I^N\coloneqq  \int_{[0,\cdot]\times \R^2} g d\nu^{N,i,k}  = I^{N,1}+I^{N,2}+I^{N,3}+I^{N,4}$$ 
with $$I^{N,l}\coloneqq \int_{[0,\cdot]\times \R^2} g d\nu^{N,i,k}_l, \qquad l\in \{1,\ldots,4\}, $$
and
$$ 
   I \coloneqq \sum_{l=1}^4I^l \coloneqq \sum_{l=1}^4\int_{[0,\cdot]\times \R^2} g d\nu^{\infty ,i,k}_l,$$
where we put $\nu^{\infty ,i,k}_4 = 0$. 

In what follows, we use Skorokhod's representation theorem. Indeed, since $D(\R_+, \R)^2\times \cP(D(\R_+, \R))$, endowed with the Skorokhod topology and the corresponding weak convergence topology, is separable, complete, and metrizable (\cite{billingsley}), there exists a probability space $(\tilde\Omega, \tilde {\cal F}, \tilde \P)$ on which are defined random variables 
$( \tilde X^{N, i , k } , \tilde \mu^N )_N $ and  $ ( \tilde X^{ i, k }, \tilde \mu )  $ which are versions of $ ( X^{N, i , k } , \mu^N ) $ and 
$ (  X^{ i, k },  \mu )  $, respectively (that is, for all $N\in\N^*$, $ {\mathcal L} ( ( X^{N, i , k } , \mu^N )) = {\mathcal L} (( \tilde X^{N, i , k } , \tilde \mu^N )) $ and $ {\mathcal L} ( ( \tilde X^{ i, k }, \tilde \mu )) = {\mathcal L} (  (  X^{ i, k },  \mu ) )$), such that the convergence 
$$ \lim_{N\to\infty}( \tilde X^{N, i , k } , \tilde  \mu^N ) =  ( \tilde X^{ i, k }, \tilde  \mu )$$
holds $\tilde\P$-almost surely.   

Write  $\tilde I^{N, l}_t$ for the respective versions of $ I^{N, l}_t,$ $l =1,2,3,4$, which are defined on $(\tilde\Omega, \tilde {\cal F}, \tilde \P)$ by
\begin{eqnarray*}
 \tilde I^{N, 1}_t &\coloneqq & \int_{0}^t f(\tilde X^{N,i}_s)  \int_{{\R^\ast} } \nu(du)  g  (\psi(\tilde X^{N,i}_s, \tilde \mu^N_s),  u/N^{1/\alpha}) ds ,\\ 
  \tilde I^{N, 2}_t & \coloneqq & \int_{0}^t f(\tilde X^{N,k}_s)  \int_{{\R^\ast} } \nu(du)  g  (  u/N^{1/\alpha},\psi(\tilde X^{N,k}_s, \tilde \mu^N_s)) ds ,\\ 
\tilde I_t^{N,3} & \coloneqq & \int_0^t\tilde\mu_s^N(f) ds\times A_N(g),\\
\tilde I^{N, 4}_t & \coloneqq & \int_{0}^t \left (f(\tilde X^{N,i}_s)+f(\tilde X^{N,k}_s)\right )ds \times A_N(g)/N,
\end{eqnarray*}
where we denote
\begin{equation*}
A_N(g) \coloneqq N\int_{\R^\ast} g(u/N^{1/\alpha},u/N^{1/\alpha})\nu(du) .
\end{equation*}
Moreover, $\tilde I^{ l}_t$ denotes the respective version of $ I^{ l}_t,$ $l=1,2,3, 4,$ given by   
\begin{eqnarray*}
 \tilde I^{ 1}_t &\coloneqq &  \int_{0}^t f(\tilde X^{i}_s)   g  (\psi(\tilde X^{i}_s, \tilde \mu_s),  0) ds ,\\
  \tilde I^{ 2}_t &\coloneqq  & \int_{0}^t f(\tilde X^{k}_s)   g  (  0,\psi(\tilde X^{k}_s, \tilde \mu_s)) ds ,\\ 
\tilde I_t^{3} &\coloneqq &\int_0^t\tilde\mu_s(f) ds\times \int_{\R^\ast} g(u,u)\nu^{\alpha}(du),\\
\tilde I^{ 4}_t &\coloneqq &0 .
\end{eqnarray*}

By Lemma \ref{lem:app4}, for $ l= 1, 2, $ and for all $ t \geq 0 ,$ almost surely, as $ N \to \infty, $ 
$$ \sup_{s \le t } | \tilde I^{N, l}_s -  \tilde I^l_s | \to 0 .
$$
Using Proposition \ref{prop:cb} with $G(x) \coloneqq g(x,x)$, for $g  \in \mathcal{C}_1 ( \R^2),$ we deduce that  
\begin{equation*}
\lim_{N\to\infty} A_N(g)=\lim_{N\to\infty}N\int_{\R^\ast} g(u/N^{1/\alpha},u/N^{1/\alpha})\nu(du)ds =\int_{\R^\ast} g(u,u)\nu^{\alpha}(du).
\end{equation*}
In particular $A_N/N\to 0.$ Relying now on Lemma \ref{lem:app2} and Lemma \ref{lem:app4} (which we apply with the choice $ g ( x, x) = 1 $ for all $x$), we have, for any $ t \geq 0, $ for $l=3, 4 ,$ almost surely as $ N \to \infty, $  
$$ \sup_{s \le t } | \tilde I^{N, l}_s -  \tilde I^l_s | \to 0 .
$$

This concludes the proof of the almost sure convergence of  $\tilde I^N$ to $\tilde I$ as processes in  $C ( \R_+ , \R )$ (with the topology of uniform convergence on compact sets). Let now  $ \Phi : D ( \R_+, \R^2) \times C ( \R_+, \R ) \to \R $ be a continuous bounded function. Then, by dominated convergence, as $ N \to \infty, $ 
$$ \E [ \Phi ( \tilde X^{N, i , k }, \tilde I^N)] \to \E [ \Phi (  \tilde X^{ i , k }, \tilde I)].$$ 
But, by equality of the laws, 
$  \E [ \Phi (  \tilde X^{N, i , k },\tilde I^N)] =  \E [ \Phi (  X^{N, i , k }, I^N)]$
and  
$ \E [ \Phi (  \tilde X^{ i , k }, \tilde I)] = \E [ \Phi (   X^{ i , k }, I)].$
So we have shown that if $N\to\infty,$
\begin{align*}
 \E [ \Phi (   X^{N, i , k }, I^N)]=&\E \left[ \Phi \left(   X^{N, i , k }, \int_{[0, \cdot ] \times \R^2 } g   d \nu^{N, i, k }  \right) \right] \\
  &\longrightarrow \E \left[ \Phi \left(    X^{ i , k }, \int_{[0, \cdot ] \times \R^2 } g   d \nu^{\infty , i, k } \right)\right]=\E [ \Phi  (  X^{ i , k }, I)],
  \end{align*}
implying the desired convergence in law \eqref{eq:cvnu}. This concludes the first step of our proof. 	    	 

\noindent
{\bf Step 2.}

According to the previous step, recalling that $h\equiv 0$, the two-dimensional limit semimartingale $X^{i,k}$ is a process of the form $ X^{ i, k } = ( X^i, X^k), $ where for $ l=i, k , $ we have the decomposition
  \begin{equation*}
X^{ l}_t(h)= X^{ l }_0+ B^{l}_t(h)+M^{l}_t(h) .
\end{equation*}
The jump measure of the process $X^{i,k}$ has as compensator the random measure $\nu^{\infty,i,k}$ which is given in \eqref{eq:def_nu^infty,i} above.  The goal of this step is to prove that we
can represent the jump part of $ X^{i, k }$ in terms of three independent driving Poisson random measures. For this sake, we rely on Theorem II.7.4 in \cite{IW} which we apply with $\mathbf{X} = \R^2 \setminus  \{(0,0)\}$ and $q(t, dz, \omega) \eqqcolon q(t,dz)$
where
\begin{multline*}
q(t,dz ) \coloneqq f(X^{i}_{t^-})\delta_{(\psi(X^{i}_{t^-}, \mu_{t^-}), 0)} (dz) + f(X^{k}_{t^-})\delta_{(0, \psi(X^{k}_{t^-}, \mu_{t^-}))} (dz) \\
	+  \int_{\R^\ast} \nu^\alpha(dy) \delta_{( (\mu_{t^-}(f))^{1/\alpha} y, (\mu_{t^-}(f))^{1/\alpha} y)}(dz) , 
\end{multline*}
so that $\nu^{\infty,i,k}(t, dz) = q(t,dz) dt$.

Let $\chi \coloneqq  ([0, \infty [ \times  \R) \setminus\{(0,0)\}. $ For any  $\xi = (\xi_1, \xi_2 ) \in \chi ,$ we define 
\begin{align*}
	\theta(t, \xi) & \coloneqq (\psi(X^{i}_{t^-}, \mu_{t^-}), 0) \indiq_{\{\xi_1 \leq f(X^{i}_{t^-}) \}} + (0, \psi(X^{k}_{t^-}, \mu_{t^-})) \indiq_{\{ \xi_1 \in ] \norme{f}_\infty ,  \norme{f}_\infty + f(X^{k}_{t^-}) ] \}}  \\ 
	& +  ((\mu_{t^-}(f) )^{1/\alpha}\xi_2, (\mu_{t^-}(f) )^{1/\alpha} \xi_2) \indiq_{\{ \xi_1 \in 2 ] \norme{f}_\infty ,  3\norme{f}_\infty  ] \}},
\end{align*}
and we introduce a measure $m $ on $ \chi $ by $m(d\xi)  \coloneqq m_1(d\xi) + m_2(d\xi) + m_3(d\xi),$ where
\begin{align*}
	m_1(d\xi)  &\coloneqq  \indiq_{\{\xi_1 \leq \norme{f}_\infty \}} d\xi_1 \indiq_{\{ \xi_2 \in [0,1] \} } d\xi_2  ,   \\
	m_2(d\xi) & \coloneqq  \indiq_{\{ \xi_1 \in ] \norme{f}_\infty ,  2\norme{f}_\infty ] \}} d\xi_1 \indiq_{\{ \xi_2 \in [0,1] \} } d\xi_2  ,  \\
	m_3(d\xi)  &\coloneqq \frac{1}{\norme{f}_\infty}\indiq_{\{\xi_1 \in 2 ] \norme{f}_\infty , 3 \norme{f}_\infty   ] \}} d\xi_1 \nu^\alpha(d\xi_2) .
\end{align*}
Notice that $ m_1, m_2, m_3 $ have disjoint supports. 

In what follows, we verify that the condition 
\begin{equation}\label{eq:conditionIW}
	m\left( \{ \xi \in \chi \, : \, \theta(t,\xi) \in E \} \right) = q(t,E)
\end{equation}
is satisfied for all $E\subset \R^2\setminus\{(0,0)\}$. 
Clearly, 
$$
	m\left( \{ \xi \in \chi \, : \, \theta(t,\xi) \in E \} \right) 
	 = \phi_1 + \phi_2 + \phi_3, 
$$
with
\begin{align*}
	\phi_1 & \coloneqq  \int_\chi \indiq_{\{ \theta(t,\xi) \in E\}} m_1(d\xi) = \int_{[0, \norme{f}_\infty] \times [0,1]} \indiq_{\{ (\psi(X^{i}_{t^-}, \mu_{t^-}), 0)  \indiq_{\{\xi_1 \leq f(X^{i}_{t^-}) \}} \in E \}} d\xi_2 d\xi_1  \\
	& = \int_{[0, f(X^{i}_{t^-})]} d\xi_1 \indiq_{\{ (\psi(X^{i}_{t^-}, \mu_{t^-}), 0) \in E\}} + \int_{]f(X^{i}_{t^-}), \norme{f}_\infty]} d\xi_1 \indiq_{\{ (0, 0) \in E\}} \\
	& = f(X^{i}_{t^-}) \indiq_{\{ (\psi(X^{i}_{t^-}, \mu_{t^-}), 0) \in E \} }  + 0,
\end{align*}
since $ E$ does not contain $ (0,0 ),$ and similarly, 
\begin{align*}
	\phi_2 & \coloneqq  \int_\chi \indiq_{\{ \theta(t,\xi) \in E\}} m_2(d\xi) \\
	& = \int_{]\norme{f}_\infty, 2 \norme{f}_\infty] \times [0,1]} 
	\indiq_{\{ (0, \psi(X^{k}_{t^-}, \mu_{t^-})) \indiq_{\{ \xi_1 \in ] \norme{f}_\infty ,  \norme{f}_\infty + f(X^{k}_{t^-}) ] \}} \in E \}} d\xi_2 d\xi_1  \\
	& = \int_{] \norme{f}_\infty , \norme{f}_\infty + f(X^{k}_{t^-})]} d\xi_1 \indiq_{\{ (0, \psi(X^{k}_{t^-}, \mu_{t^-})) \in E\}} + \int_{] \norme{f}_\infty +  f(X^{i}_{t^-}), 2 \norme{f}_\infty]} d\xi_1 \indiq_{\{ (0, 0) \in E\}} \\
	& = f(X^{k}_{t^-}) \indiq_{\{ (0, \psi(X^{k}_{t^-}, \mu_{t^-})) \in E \} },
\end{align*}
and finally
\begin{align*}
	\phi_3 & \coloneqq  \int_\chi \indiq_{\{ \theta(t,\xi) \in E\}} m_3(d\xi) \\
	&= \frac{1}{\norme{f}_\infty} \int_{2]\norme{f}_\infty, 3 \norme{f}_\infty  ] \times \R^\ast} d \xi_1 \nu^\alpha(d\xi_2) \indiq_{\{ ((\mu_{t^-}(f))^{1/\alpha} \xi_2, (\mu_{t^-}(f))^{1/\alpha} \xi_2)  \in E \}  } \\
	& = \int_{\R^\ast} \nu^\alpha(d\xi_2) \indiq_{\{ ((\mu_{t^-}(f))^{1/\alpha} \xi_2, (\mu_{t^-}(f))^{1/\alpha} \xi_2)  \in E \}  } .
\end{align*}

Once condition \eqref{eq:conditionIW} is checked, Theorem II.7.4 in \cite{IW} yields the existence of a Poisson random measure $N$ on $\R_+ \times \chi$  with intensity $dt  m(d\xi),$ such that we have the representation 
\begin{equation*}
	\sum_{ s \le t  }  \indiq_{ \{\Delta X_s^{i, k } \in E \}} = \int_{[0,t] \times \chi} \indiq_{\{\theta(t,\xi) \in E \}} N(ds, d\xi) ,
\end{equation*}
for any $E\subset \R^2\setminus\{(0,0)\}.$

\noindent
{\bf Step 3.} 

We conclude our representation by constructing three independent Poisson random measures $\bar\pi^{i}(ds, dz)$, $\bar\pi^{k}(ds, dz)$ and $M(ds, dz)$ on $\R_+ \times [0,\norme{f}_\infty]$, $\R_+ \times [0,\norme{f}_\infty]$ and on $\R_+ \times \R^\ast$ respectively, with intensities $ds   dz $,  $ds  dz $ and $ds   \nu^\alpha(dz)$ respectively, such that 
$$\sum_{s\leq t}\Delta X_s^{i,k}=I^{\infty, i}+I^{\infty, k}+J^\infty ,$$ with
\[
	I^{\infty, l} = \int_{[0,t] \times \R_+} \psi(X^{l}_{s^-}, \mu_{s^-}) \indiq_{\{ z\leq f(X^{l}_{s^-}) \}} \bar\pi^{l}(ds, dz), \qquad\qquad l\in \{i, k\},
\]
and 
\[
	J^\infty = \int_{[0,t] \times \R^\ast} (\mu_{s^-}(f))^{1/\alpha} z M(ds,dz) ,
\]
to obtain the desired representation. 

Let us write $ \theta ( t , \xi ) = ( \theta_1 (t , \xi ), \theta_2 ( t, \xi) , \theta_3 ( t, \xi ) ) .$ Then 
\begin{multline*}
I^{ \infty , i}_t = 	\int_{[0,t] \times \chi } \theta_1(s,\xi) N(ds,d\xi) \\
= \int_{[0,t] \times [0, \norme{f}_\infty]} \psi(X^{i}_{s^-}, \mu_{s^-}) \indiq_{\{ \xi_1 \leq f(X^{i}_{s^-}) \}} N(ds,d\xi_1,  [ 0, 1 ] ) ,
\end{multline*}
\begin{multline*}
I^{ \infty , k }_t = 	\int_{[0,t] \times \chi } \theta_2(s,\xi)  N(ds,d\xi) \\
	 = \int_{[0,t] \times ] \norme{f}_\infty,+\infty [ } \psi(X^{k}_{s^-}, \mu_{s^-}) \indiq_{\{ \norme{f}_\infty < \xi_1 \leq \norme{f}_\infty + f(X^{k}_{s^-}) \}} N(ds,d\xi_1, [ 0, 1 ] ) ,
\end{multline*}
and finally 
$$	
	J^\infty_t = 	\int_{[0,t] \times \chi } \theta_3(s,\xi) N(ds,d\xi) 
			 = \int_{[0,t]\times \R^\ast}  (\mu_{s^-}(f))^{1/\alpha} \xi_2 N(ds,  ] 2\norme{f}_\infty, 3 \norme{f}_\infty ], d\xi_2) .
$$
Therefore, we define for any $A \subset [0, \norme{f}_\infty],$ 
\[
	\bar\pi^{i}([0,t]\times A ) \coloneqq N([0,t]\times A \times [0, 1 ] ) ,
\]
\[
	\bar\pi^{k}([0,t]\times A ) \coloneqq N([0,t] \times (\norme{f}_\infty + A )  \times [0, 1 ] ) ,
\]
where $ c + A = \{ c + a : a \in A\},$ 
and, for any $A \subset \R^\ast$, we set 
\[
	M([0,t]\times A ) \coloneqq N([0,t]\times  ] 2\norme{f}_\infty, 3 \norme{f}_\infty ] \times A) .
\]
$\bar\pi^{i}$, $\bar\pi^k$ and $M$ are independent by construction, since they use disjoint subsets of the support of $N.$ Extending $ \bar \pi^i $ and $ \bar \pi^k $ to the whole space $ \R_+ \times \R_+ $ such that their independence is preserved then yields the desired representation.  
\end{proof}
 
\subsection{Proof of Theorem \ref{theo:main-repr} for  $ \alpha > 1 $}\label{sec:passage_to_lim:alpha>1}
\begin{proof}
The proof follows the same ideas as the proof of the case $\alpha <1,$  but has to be adapted to the case $ \alpha > 1 .$  
As in the proof of the case $\alpha <1,$ it is immediate to see that $(X^{N, i }, X^{ N, k}, \mu^N )_N $  is tight since its components are. 
In what follows, we therefore consider a fixed subsequence that converges in law to a limit $(X^{  i }, X^{  k}, \mu )$
and identify the structure of the limit $(X^{  i }, X^{  k})$  relying on Theorem IX.2.4 in \cite{JS}.

\noindent{\bf Step 1.} 

We first apply Theorem II.2.34 of \cite{JS} to the semimartingale $(X^{N, i }, X^{ N, k})_N$ , with a fixed continuous truncation function $h$ such that $ h( x) = x $ for all $ | x| \le 1 $ and $ h( x) = 0 $ for all $ |x| > 2.$ For this  function $h,$ with the notation of equation II.2.4 in \cite{JS}, we set
  \begin{equation*}
X^{N, l}_t(h)= X^{N, l }_0+ B^{N,l}_t(h)+M^{N,l}_t(h),\;\;\;  l=i,k ,
\end{equation*}
where for $l=i, k,$ 
\[ M(h)^{N,l}_t=\sum_{j\neq l} \int_{[0,t]\times\R_+\times\R^\ast} h\left(\frac u{N^{1/\alpha}}\right)\indiq_{\{z\leq f(X^{N,j}_{s^-})\}}\tilde \pi^j(ds,dz,du),\]
where $\tilde\pi^{j} (ds, dz, du ) = \pi^{j} (ds, dz, du ) - ds dz \nu ( du ) $ is the compensated measure, and
\begin{align*}
B_t^{N,l}(h)=&\int_0^tb(X_s^{N,l},\mu_s^N)ds+\int_0^t N \mu_s^N ( f) ds\int_{ \R^\ast} \nu (du ) h ( u/ N^{ 1/\alpha } ) \\
&-\int_0^tf(X_s^{N,l}) ds\int_{ \R^\ast} \nu (du ) h ( u/ N^{ 1/\alpha } ).
\end{align*}
This implies (see Definition II.2.6 in \cite{JS}) that the characteristics of $X^{N,i,k}$ are $(B^{N,i,k},C^{N,i,k}, \nu^{N,i,k})$ where

$$B_t^{N,i,k}=(B_t^{N,i}(h),B_t^{N,k}(h)), \qquad C^{N,i,k}=0,$$
and
\begin{align*}
	\nu^{N, i, k }(ds, dx_1,dx_2)&= f(X^{N,i}_s)ds\int_{\R^\ast} \nu(du) \delta_{( 0,u/N^{1/\alpha})}(dx_1, dx_2)\\
	& + f(X^{N,k}_s)ds\int_{\R^\ast} \nu(du) \delta_{( u/N^{1/\alpha}, 0)}(dx_1, dx_2)\\
	& + \mu^N_s ( f) ds  \int_{\R^\ast} \nu(du) N \delta_{( u/N^{1/\alpha},u/N^{1/\alpha})}(dx_1,dx_2)  \\ 
	&- \left(  f(X^{N,i}_s) +  f(X^{N,k}_s)  \right) ds \int_{\R^\ast} \nu(du)  \delta_{( u/N^{1/\alpha},u/N^{1/\alpha})}(dx_1,dx_2).
\end{align*}
Moreover, the second modified characteristic (see Equation 2.2 of Chapter IX in \cite{JS}) is the $ 2 \times 2 $ matrix with entries 
\begin{eqnarray*}
( \tilde C^{N, i,k}_t)_{l,m} & = & \int_{[0,t]\times \R^2} h(x_l) h(x_m) \nu^N(ds, dx_1, dx_2) 
  \\
 & = &
 \int_0^t   \mu^N_s ( f) ds  N \int_{\R^\ast} h^2 ( u/N^{1/\alpha}) \nu ( du)\\
 && -\int_0^t   \left(  f(X^{N,i}_s) +  f(X^{N,k}_s)\right ) ds   \int_{\R^\ast} h^2 ( u/N^{1/\alpha}) \nu ( du) ,
\end{eqnarray*}
for $ 1 \le l, m \le 2,$ where we have used that $ h(0) =0 .$

To employ Theorem IX.2.4 in \cite{JS}, we will show  that, as $N\to\infty,$ 
\begin{equation}\label{eq:tildec}
(X^{N,i,k} ,B^{N,i,k} ,\tilde C^N  )\convlaw (X^{i,k} ,B^{i,k} ,\tilde C )
\end{equation}
in $D(\R_+, \R^{2+2+4}),$ as well as the convergence 
\begin{equation*}
 \left(X^{N,i,k}, \int_{[0,\cdot]\times \R^2} g d\nu^{N, i, k }\right) \convlaw \left(X^{i,k}, \int_{[0,\cdot]\times \R^2} g  d\nu^{\infty, i, k }\right),
 \end{equation*}
in $D(\R_+, \R^{2+1}),$ for all $g \in {\mathcal C}_1 ( \R^2), $ where 
\begin{eqnarray*}
\nu^{\infty, i, k }(ds,dx_1, dx_2) &\coloneqq  & ds \int_{\R^\ast} \nu^\alpha(dy) \delta_{ (\mu^{1/\alpha}_s(f) y, \mu^{1/\alpha}_s(f) y)}(dx_1, dx_2)  ,\\
 B_t^{i,k}(h) &\coloneqq &(B_t^{i}(h),B_t^{k}(h)),\\
B_t^{l}(h) &\coloneqq & \int_0^tb(X_s^{l},\mu_s)ds+\int_0^t \mu_s ( f) ds\int_{ \R^\ast} (h(u) - u) \nu^{\alpha} (du ) ,\qquad l=i,k ,
\end{eqnarray*}
and
$$ (\tilde C_t)_{l, m } \coloneqq \int_0^t   \mu_s ( f) ds   \int_{\R^\ast} h^2 ( u) \nu^{\alpha} ( du) , \qquad 1 \le l, m \le 2.$$ 
To prove this convergence, we start observing that for all $ g \in {\mathcal C}_1 ( \R^2),$ 
\begin{multline*}
 \int_{[0,\cdot]\times \R^2 } g d\nu^{N,i,k} =  A_N(g)  \int_0^\cdot \mu_s^N ( f) ds \\
 +\int_0^\cdot  f(X^{N,i}_s) ds \left( \int_{\R^\ast} \nu ( du) \left[ g ( 0, \frac{u}{N^{ 1/\alpha}} ) - g( \frac{u}{N^{ 1/\alpha}}, \frac{u}{N^{ 1/\alpha} }) \right] \right) \\
 + \int_0^\cdot  f(X^{N,k}_s) ds \left( \int_{\R^\ast} \nu ( du) \left[ g (  \frac{u}{N^{ 1/\alpha}}, 0 ) - g( \frac{u}{N^{ 1/\alpha}}, \frac{u}{N^{ 1/\alpha}} ) \right] \right)  ,
\end{multline*}
where we denote as before
$$A_N(g) \coloneqq N \int_{\R^\ast} \nu (du ) g ( u/N^{1/ \alpha }, u/N^{1/ \alpha }).$$
Equation \eqref{eq:js1} implies that for  $g \in {\mathcal C}_1 ( \R^2), $ as $N \to \infty, $  
$$A_N(g )\to  \int_{\R\ast}g (u)\nu^{\alpha}(du).$$
Moreover, by dominated convergence, using that $g$ is continuous and  $ g(0,0 ) =0 ,$ we also have that, as $ N \to \infty, $ 
$$  \int_{\R^\ast} \nu ( du) \left[ g ( 0, \frac{u}{N^{ 1/\alpha}} ) - g( \frac{u}{N^{ 1/\alpha}}, \frac{u}{N^{ 1/\alpha} }) \right]  \to 0 .$$

Using the Skorokhod representation theorem to represent $(X^{N,k}, \mu^N)$, exactly as in the proof for $\alpha <1,$ we deduce from the above discussion the weak convergence of  
$(X^{N,i,k} , \int_{[0,\cdot]\times \R} g d\nu^{N, i,k })_N \convlaw (X^{i,k} , \int_{[0,\cdot]\times \R} g  d\nu^{\infty, i,k })$ in the Skorokhod space, for any $g \in {\mathcal C}_1 ( \R^2 ).$ 

Finally, to prove the convergence \eqref{eq:tildec}, notice that, by dominated convergence, 
$$ \int_{\R^\ast}\left( h( u/N^{1/\alpha} ) + h^2 ( u/N^{1/\alpha} ) \right) \nu ( du) \to 0,$$ 
such that the desired convergence follows from the second item of Proposition \ref{prop:cb}, together with the arguments already used in the case $ \alpha < 1 .$

\noindent
{\bf Step 2.}

Due to Theorem 2.4, Chapter IX of \cite{JS}, the limit process $X^{i,k}=(X^{i},X^k)$ is a semimartingale with characteristics $(B^{i, k },0,\nu^{ \infty, i, k }).$
Following the steps of the case $\alpha <1,$ we deduce that there exists a Poisson random measure $M$ on $\R_+ \times \R^\ast $  with intensity $dt  \nu^\alpha ( dz) ,$  such that we have the representation 
\begin{equation*}
	\sum_{ s \le t  }  \indiq_{ \Delta X^l_s   \in E }   = \int_{[0,t] \times \R^\ast } \indiq_{\{\theta(t,\xi) \in E \}} M(ds, d \xi ) ,  \qquad l = i, k, 
\end{equation*}
for any $ E $ such that $ \bar E \subset \R^\ast ,$ where $\theta(t, \xi) =  (\mu_{t^-}(f))^{1/\alpha}  \xi  .$ 
The desired representation of $ X^l, l=i, k , $ follows from this as in the proof of the case $\alpha<1. $
\end{proof}

\section{Proof of Theorem \ref{theo:main-ident}}\label{sec:uniqueness}

The results and conclusions in the present section hold for both cases $\alpha<1$ and $\alpha>1$. 
\begin{proposition}\label{thm:dir_m}
The directing measure of the  system \eqref{eq:limit:system} is  $\mathcal{L}(\bar X^1 | S^\alpha)$.
\end{proposition}

\begin{proof}
By Lemma (2.12)(a) of \cite{aldous}, it suffices to show that, for all $N\geq 1$, $(\bar X^i)_{i=1}^N$ are conditionally i.i.d. given $S^\alpha$.
By the definition of a strong solution, there exists a measurable map $\Phi$, such that for all $i\in\N^*$, all $T\geq 0,$
    \[
    	(\bar X^i_t)_{t \in [0,T]} = \Phi(\bar X^ i_0, \bar\pi^i_{| [0,T]\times \R_+}, (S^\alpha_t)_{t \in [0,T]}) .
    \]
Let now $ n \geq 2 $ and take a family of continuous bounded real-valued functions $g_i, 1 \le i \le n. $ Using that $S^\alpha$ is independent of the Poisson random measures and of the initial conditions, we have that   
    \begin{multline*}
    	\E\left[\prod_{i=1}^n g_i (\bar X^i_{[0,T]})  | S^\alpha \right]  
	= \E\left[ \prod_{i=1}^n g_i(\Phi(\bar X^{i}_0, \bar\pi^i_{| [0,T]\times \R_+}, (S^\alpha_t)_{t \in [0,T]}))  | S^\alpha \right]\\
	 = \prod_{i=1}^n  \E\left[ g_i(\Phi(\bar X^{i}_0, \bar\pi^i_{| [0,T]\times \R_+}, (S^\alpha_t)_{t \in [0,T]})) | S^\alpha\right]  
	 =\prod_{i=1}^n  \E \left[g_i(\bar X^{i})  | S^\alpha \right]   ,
    \end{multline*}
which concludes the proof. 
\end{proof}

We now give the proof of Theorem \ref{theo:main-ident}. 

\begin{proof}
For some fixed $ K > 0,  $ let $T_K=\inf\{t\geq 0 \, | \, \Delta S^{\alpha}_t\geq K\}.$ Recall that $M(ds,dz)$ denotes the jump measure of $S^\alpha$. To compare $ \bar X^i $ and $ X^i, $ we decompose the processes into their stochastic integral part, their continuous drift part and, in case $ \alpha < 1, $ their bounded variation main jump part. To do so, we introduce the following notation. 

For generic processes $X$ and $\tilde X$ and measures $\mu$ and $\tilde \mu$ that can be either the conditional laws of the processes given $S^\alpha,$ their directing measures or the empirical measures, we define
\begin{equation*}
			M_t(\mu,\tilde \mu) \coloneqq 
	\begin{cases}
	\int_{[0, t ] \times \R^\ast }   \left( \mu^{1/ \alpha}_{s^-}(f)  - \tilde \mu^{1/ \alpha}_{s^-}(f)\right) z    M ( ds, dz)  ,& \alpha < 1 \\
	\int_{[0, t ] \times \R^\ast }   \left(\mu^{1/ \alpha}_{s^-}(f)  - \tilde \mu^{1/ \alpha}_{s^-}(f)\right) z    \tilde M ( ds, dz), & \alpha > 1  ,
	\end{cases}
\end{equation*}
\begin{equation*}
			B_t (X, \tilde X, \mu , \tilde \mu) \coloneqq \int_0^t[   b ( X_{s} ,  \mu_{s}) -  b ( \tilde X_{s}, \tilde \mu_s  )] ds ,\;   \alpha<1 \;\,\mbox{and}\;\, \alpha >1,
\end{equation*}
and $\Psi_t= \Psi_t (X, \tilde X, \mu , \tilde \mu)$ as
\begin{equation*}
			 \Psi_t  \coloneqq 
	\begin{cases}
\int_{[0,t]\times \R_+} \left[   \psi(X_{s^-},\mu_{s^-}) \indiq_{\{z\leq f(X_{s^-})\}} - \psi(\tilde X_{s^-},\tilde\mu_{s^-}) \indiq_{\{z\leq f(\tilde X_{s^-})\}} \right] \bar\pi(ds,dz),& \alpha < 1 \\
	0, & \alpha > 1  .
	\end{cases}
\end{equation*}

This notation coincides with the notation of Lemmas 5.1, 5.2 and 5.3 in \cite{LLM}.

\noindent

\noindent{\bf  The case $\pmb{\alpha >1}$.}

Introducing the empirical measures $\mu^{N,X}$ and $\mu^{N,\bar X}$ of the first $N$ coordinates of the systems $(X^i)_{ i \geq 1}$ and $(\bar X^i)_{ i \geq 1 }, $ respectively, we can then upper bound  
\begin{multline*}
| X_t^i - \bar X_t^i | \le | M_t ( \mu , \mu^{N, X})| + | M_t ( \mu^{N, X}, \mu^{N, \bar X} ) | + | M_t ( \mu^{N, \bar X}, \bar \mu )| \\
+| B_t ( X^i, X^i, \mu, \mu^{N, X} ) | + | B_t ( X^i, \bar X^i, \mu^{N, X},  \mu^{N, \bar X} ) | + | B_t ( \bar X^i, \bar X^i,   \mu^{N, \bar X}, \bar \mu ) | .
\end{multline*} 

By Lemma 5.1 of \cite{LLM} and the fact that $f$ is Lipschitz, we have that for any $ 1< \alpha < \alpha_+ ,$ any two measures  $\mu$ and $\tilde \mu$ and for some constant that depends on $  \alpha, \alpha_+ $ and on $K ,$ 
\begin{equation*}
 \E\left [  |M_t(\mu,\tilde \mu)|^{\alpha_+}\indiq_{\{t < T_K\}}\right ]  \le C     
 \begin{cases}
   \int_0^t \E \left [   W_1^{\alpha_+}(  \mu_{s}, \tilde\mu_{s} )\indiq_{\{s < T_K\}}\right] ds \\
   \int_0^t \E \left [   W_1(  \mu_{s}, \tilde\mu_{s} )\indiq_{\{s < T_K\}}\right] ds  .
\end{cases} 
\end{equation*} 
We will use the second upper bound when comparing $ \mu $ to its empirical version $ \mu^{N, X}, $ and $ \bar \mu $ to $ \mu^{N, \bar X}.$ But we will use the first upper bound when comparing the two empirical measures. 

Moreover, Lemma 5.2 of \cite{LLM} implies that in case $ \alpha > 1,$  
\begin{multline*} 
\E[|B_t (X, \tilde X, \mu , \tilde \mu)|^{\ap}\indiq_{\{t < T_K\}}]\leq   
  C_t   
  \begin{cases}
\E \int_0^t  \indiq_{\{s < T_K\}} \left(  |  X_s - \tilde X_s |^{\ap} + W_{1}^{\ap} ( \mu_s, \tilde \mu_s ) \right) ds \\
\E \int_0^t  \indiq_{\{s < T_K\}} \left( | X_s - \tilde X_s | + W_{1} ( \mu_s, \tilde \mu_s )\right) ds .
\end{cases}
\end{multline*}
We will use the second upper bound when dealing with $| B_t ( X^i, X^i, \mu, \mu^{N, X} ) |$ and with $| B_t ( \bar X^i, \bar X^i,   \mu^{N, \bar X}, \bar \mu ) |,$ while we use the first upper bound when dealing with $| B_t ( X^i, \bar X^i, \mu^{N, X},  \mu^{N, \bar X} ) |.$ 
Doing so, we get 
\begin{multline}\label{al:gron} 
\E[| X^{i}_t - \bar X^{i}_t |^{\ap}	\indiq_{\{t < T_K\}}]  \leq C_t \int_0^t \E[\indiq_{\{s < T_K\}}\left (| X^{i}_s - \bar X^{i}_s |^{\ap}+W_1^{\ap}(\mu_s^{N, X}, \mu_s^{N, \bar X} ) \right) ] ds  \\
+C_t \int_0^t \E[ W_{1}(\mu_s,\mu_s^{N, X} ) + W_1 ( \bar\mu_s, \mu_s^{N, \bar X } )] ds . 
\end{multline}

Using Jensen's inequality with the convex function  $x\to x^{\ap}$ and the diagonal coupling $\frac 1N\sum_{i=1}^N\delta_{(X^{i}_s,\bar X^{i}_s)}$ of $\mu_s^{N, X}$ and $\mu_s^{N, \bar X},$ we obtain
\begin{align*}
	\E[ \indiq_{\{s < T_K\}}& W_{1}^{\ap}( \mu^{N, X}_s, \mu^{N, \bar X}_s)]   \leq \E\left[ \indiq_{\{s < T_K\}}\left( \frac{1}{N}\sum_{j=1}^N |X^{j}_s - \bar X^{j}_s|\right )^{\ap}\right]  \\
	&\leq \E\left[ \indiq_{\{s < T_K\}} \frac{1}{N}\sum_{j=1}^N |X^{j}_s - \bar X^{j}_s|^{\ap}\right] =  \E [ \indiq_{\{s < T_K\}} |X^{i}_s - \bar X^{i}_s|^{\ap} ] .\nonumber
\end{align*}
As a consequence, \eqref{al:gron} now reads as 
\begin{multline}\label{al:gron2} 
\E[| X^{i}_t - \bar X^{i}_t |^{\ap}	\indiq_{\{t < T_K\}}]  \leq C_t \int_0^t \E[\indiq_{\{s < T_K\}} | X^{i}_s - \bar X^{i}_s |^{\ap}  ] ds  \\
+C_t \int_0^t \E[ W_{1}(\mu_s,\mu_s^{N, X} ) + W_1 ( \bar\mu_s, \mu_s^{N, \bar X } )] ds.
\end{multline}
The last two terms above evaluate the distance between an empirical measure and its limiting directing measure for an exchangeable system. These two terms are
 treated using Lemma 5.7 in \cite{LLM}, which is based on the results of \cite{fournierguillin}. It yields
\begin{align*}
\E[ W_{1}(\mu_s,\mu_s^{N, X} ) + W_1 ( \bar\mu_s, \mu_s^{N, \bar X } )]   \leq C_s N^{- 1/2} ,
\end{align*}
such that all in all, \eqref{al:gron2}  becomes 
\begin{equation*}\label{eq:Gron}
\E[| X^{i}_t - \bar X^{i}_t |^{\ap}	\indiq_{\{t < T_K\}}]  \leq C_t \int_0^t\E[\indiq_{\{s < T_K\}}| X^{i}_s - \bar X^{i}_s |^{\ap}] ds +C_t N^{-1 /2}.
\end{equation*}
Using Gr\"onwall's lemma, we conclude that there exists a different constant, still denoted $C,$ which depends on $\alpha, \ap , K $ and $ t,$ such that 
\begin{equation*}\label{eq:Gron1}
\E[| X^{i}_t - \bar X^{i}_t |^{\ap}	\indiq_{\{t < T_K\}}] \leq  C N^{- 1/2} ,
\end{equation*}
such that, letting $N\to\infty, $ for any fixed $t, $ almost surely, 
$$| X^{i}_t - \bar X^{i}_t |^{\ap}	\indiq_{\{t < T_K\}}=0\; a.s.$$
Letting $K\to \infty$ and using that $\lim_{K\to\infty}T_K=\infty,$ this implies the claim in case $ \alpha > 1.$

\noindent

\noindent{\bf  The case $\pmb{\alpha <1}$.}

We now discuss the case $\alpha <1.$ By Lemmas 5.1, 5.2 and 5.3 of \cite{LLM}, recalling that $d_{\am}(x,y)\coloneqq |x-y|\wedge|x-y|^{\am},$ for a constant depending on $ \alpha $ and $K,$ 
 \begin{equation*}\label{eq:controlI2}
 \E\left [  |M_t(\mu,\tilde \mu)|\indiq_{\{t < T_K\}}\right ]  \le C
 \begin{cases}
   \int_0^t\E\left[|\mu_s(f) - \tilde \mu_s(f)|\indiq_{\{s < T_K\}}\right]ds \\
 \int_0^t\E\left[W_{d_\am}(\mu_s,\tilde \mu_s)\indiq_{\{s < T_K\}}\right]ds ,
\end{cases}  
 \end{equation*}
\begin{equation*}
\E[|B_t (X, \tilde X, \mu , \tilde \mu)|\indiq_{\{t < T_K \}}]\leq C 
\E\int_0^t \indiq_{\{s<T_K\}} (d_\am(X_s,\tilde X_s)+W_{d_\am}(\mu_s,\tilde\mu_s))ds ,
\end{equation*} 
and finally, 
\begin{equation*}\label{eq:controlpsi1}
\E[|\Psi_t(X,\tilde X, \mu, \tilde \mu)|\indiq_{\{t < T_K\}}] \leq C \int_0^t \E\left[ ( d_\am(X_s,\tilde X_s) + W_{d_\am}(\mu_s, \tilde \mu_s)) \indiq_{\{t<T_K\}}\right] ds.
\end{equation*}
Therefore, 
\begin{multline*}
	\E[| X^{i}_t - \bar X^{i}_t |	\indiq_{\{t < T_K\}}]  \leq 
	\E[| B_t(X^{i}, \bar X^{i}, \mu, \bar \mu) | \indiq_{\{t < T_K\}}] + \E[| \Psi_t(X^{i}, \bar X^{i}, \mu, \bar \mu) | \indiq_{\{t < T_K\}}] \\ 
 + \E[| M_t(X^{i}, \bar X^{i}, \mu, \bar \mu) | \indiq_{\{t < T_K\}}]\\ 
  \le C \E\int_0^t \left( 
\indiq_{\{s<T_K\}}(|X^{i}_s-\bar X^{i}_s|+W_{d_{\am}}(\mu_s,\bar\mu_s)) \right) ds   .
\end{multline*}
It is straightforward to see that 
 $$W_{d_{\am}}(\mu_s,\bar\mu_s)\leq W_{\am}(\mu_s,\bar\mu_s)\wedge W_{1}(\mu_s,\bar\mu_s),$$
such that  
\begin{align}\label{in:triangle2}
	\E[\indiq_{\{s < T_K\}} W_{d_{\am}}(\mu_s, \bar\mu_s)]&  \leq \E[W_{{\am}}(\mu_s, \mu^{N,X}_s)] + \E[\indiq_{\{s < T_K\}} W_{1}(\mu^{N,X}_s, \mu^{N,\bar X}_s)] \\ \nonumber
	& + \E[ W_{{\am}}(\mu^{N,\bar X}_s, \bar\mu_s)] .
\end{align}
For the first and third terms, we use again Lemma 5.7 in \cite{LLM} for $\alpha<1$, which gives
\begin{align*}\label{in-fg}
\E[ W_{{\am}}(\mu_s, \mu^{N,X}_s) + W_{\am}(\mu^{N,\bar X}_s, \bar\mu_s)  ]  \leq C_s ( N^{-\am} \indiq_{\{\am < 1/2\}} + N^{-1/2} \indiq_{\{1/2 < \am < 1 \}}) .
\end{align*}
The term appearing in the middle of \eqref{in:triangle2} is bounded as previously by
\[\E[\indiq_{\{s < T_K\}} W_{1}(\mu^{N,X}_s, \mu^{N,\bar X}_s)] \leq \E[\indiq_{\{s < T_K\}}|X^{i}_s-\bar X^{i}_s|].\]
The conclusion then follows as in the case $ \alpha > 1, $ by Gr\"onwall's lemma. 

Let us finally show that 
$\mu_s=\bar\mu_s$ a.s. Indeed, using Lemma 2.15 of \cite{aldous}, a directing measure of an infinite-exchangeable system is an a.s. weak limit of its empirical measures. As for all $i\in\N^\ast,$  $t\geq 0,$ $X^{i}_t=\bar X^{i}_t$ a.s., the  empirical measures of both systems are equal and thus their limits too. 

\end{proof}
%
%

\section{Appendix}

In this section, we gather several technical auxiliary results that have been used throughout our proofs. Recall that ${\mathcal C}_1(\R^d)$ is defined in \cite{JS} VII.2.7.

\begin{proposition}\label{prop:cb} 
Let $\alpha\in (0,2) \setminus\{1\}.$ Then the following convergence results hold. 

\begin{enumerate}
\item 
For all $G \in \mathcal{C}_1 ( \R) $ we have that 
\begin{equation}\label{eq:js1}
\lim_{N\to\infty}N\int_{\R^\ast}G\left(\frac u{N^{1/\alpha}}\right)\nu(du)=\int_{\R^\ast}G(y)\nu^{\alpha}(dy).
\end{equation}
\item
In the case $\alpha>1$, for all bounded and continuous truncation functions $ h : \R \to \R $ such that $ h( x) = x $ in a neighborhood of $0,$ we have that 
\[\lim_{N\to\infty}N\int_{\R^\ast}h\left(\frac u{N^{1/\alpha}}\right)\nu(du)=\int_{\R^\ast}(h(y)- y) \nu^{\alpha}(dy)\]
and
\[\lim_{N\to\infty}N\int_{\R^\ast}h^2 \left(\frac u{N^{1/\alpha}}\right)\nu(du)=\int_{\R^\ast}h^2(y)\nu^{\alpha}(dy).\]
\end{enumerate}
\end{proposition}

\begin{proof}
Let $ P^N_t$ be a rate $N$ Poisson process and $ U_n, n \geq 1, $ be i.i.d. $ \sim \nu, $ independent of $ (P^N_t)_t.$ Introduce 
$$ X^N_t = \frac{1}{N^{1/\alpha}}\sum_{i=1}^{P_t^N}  U_i.$$ 
This is a compound Poisson process with jump rate $N$ and with jump distribution $ \int \nu( du ) \delta_{ u / N^{ 1/ \alpha } } (\cdot ). $ Using Corollary 16.17 in \cite {kallenberg}
we know that $ X^N $ converges in law to $ S^\alpha.$ So we are in the setting of condition 2.48, Chapter VII of \cite{JS} such that we may apply Theorem 2.52 of the same chapter. We take the truncation function $ h( x) = 0 $ if $ \alpha < 1 , $ and any fixed continuous and bounded truncation function in case $ \alpha > 1.$ In this setting, the compensator of the jump measure of $ X^N$ is given by 
$$ \nu^N (dt, dx) = N dt  \int \nu( du ) \delta_{ u / N^{ 1/ \alpha } } (dx )$$ 
and, for $ \alpha < 1, $ $ B_t^N = C_t^N = \tilde C_t^N = 0, $ whereas for $ \alpha > 1,$ 
$$ B^N_t  = N t \int  h\left( \frac{u}{N^{1/\alpha }} \right) \nu (du) , \qquad \tilde C^N_t = N t \int  h^2\left( \frac{u}{N^{1/\alpha }} \right) \nu (du) , $$
for all $ t \geq 0.$ It is immediate to check that condition 2.53 of Chapter VII of \cite{JS} holds as well. 
Then item a) of Theorem VII.2.52 \cite{JS} implies that, since we know that we have convergence in law of $ X^N $ to $S^\alpha, $ automatically 
	$$\int_{[0,t]\times \R} G(x) \nu^N(ds,dx) \to \int_{[0,t]\times \R} G(x) ds \nu^\alpha(dx) = t \int_\R G(x) \nu^\alpha ( dx) \qquad \forall \, t \in \R_+,$$
for all $G\in {\mathcal  C}_1(\R)$ (this is condition $[ \delta_{3.1}\!-\!D] $ with $ D = \R_+ $), since $ds \nu^\alpha (dz) $ is the third characteristics of the canonical triple of $S^\alpha$. 

Now we notice that, when taking $t= 1,$ we have that the expression appearing in the left-hand side of \eqref{eq:js1} equals  
$$N\int_{\R^\ast} G(u/N^{1/\alpha})\nu(du) =  \int_{[0,1]\times \R} G(x) \nu^N(ds,dx).$$ 
Therefore, we deduce from this the convergence of this expression to 
$ \int G(x)  \nu^\alpha ( dx) ,$ as claimed.

Let us now discuss the case $ \alpha > 1.$ In this case, the convergence in law of $X^N $ to $S^\alpha $ implies also that, thanks to conditions $[ \beta_{3}\!-\!D] $ and $[ \gamma_{3}\!-\!D] $ stated in 2.51 of Chapter VII. of \cite{JS}, we have convergence of $ B^N_t $ and of $ \tilde C^N_t $ to their corresponding limit quantity. These limit quantities are  
$$  B^\alpha_t  =  t \int [  h(u) - u ]  \nu^\alpha  (du) , \qquad \tilde C_t = t \int  h^2 ( u ) \nu^\alpha  (du) . $$
Taking e.g., $ t = 1 $ as in the first part of the proof, the assertion follows. 
\end{proof}

The following result, showing that weak limits of $ \mu^N$ are supported by trajectories without fixed time of discontinuity, is crucial for our convergence analysis. 
\begin{lemma}\label{lem:supp_mu}
	Grant all assumptions of Theorem \ref{theo:main}.  Let $\mu$ be the weak limit of any convergent subsequence of $(\mu^N)_N$. Then, for any fixed $t \in \R_+$, a.s., 
	\[ 
		\mu(\{ \gamma \in D(\R_+, \R) \, : \, |\Delta \gamma(t)| \neq  0 \}) = 0  . 
	\]
\end{lemma}

\begin{proof}
Throughout this proof we write $ P^N = \mathcal L ( \mu^N) $  and $P = \mathcal L ( \mu). $ The proof is by contradiction. 
So let us suppose  that there exists $t >0$ such that 
\[
    P( \{  \mu \in \cP ( D(\R_+, \R) ) :   \mu(\{ \gamma  \, : \, |\Delta \gamma(t)| \neq  0 \} ) \}  ) > 0 . 
\]
This implies that there exist $a,\,b >  0$ such that the event $E \coloneqq \{ \mu \in \cP ( D(\R_+, \R)) :  \mu(\{ \gamma \, : \, |\Delta \gamma(t)| > a \} ) > b \}$ has positive probability, that is, 
$$ P ( E) > 0.$$  

Now, fix any $\epsilon >0$ 
and define the set of paths $B^\epsilon_a \in \mathcal{B}(D(\R_+, \R))$ as
\[
	B^\epsilon_a \coloneqq \{ \gamma \in D(\R_+, \R) : \sup_{s \in (t-\epsilon, t+\epsilon)} |\Delta \gamma(s)| > a \} \subset D(\R_+, \R).
\]
By construction, $B^\epsilon_a$ is an open subset of $D(\R_+, \R),$ endowed with the Skorokhod topology, which can be seen as follows. 
Let $\gamma  \in B^\epsilon_a. $ Then there exist $ s \in (t-\epsilon, t+\epsilon)$ and $\eta >0$ such that $|\Delta \gamma ( s)| \geq a + \eta .$ Consider now a sequence $\gamma^n \to \gamma $ in $D(\R_+, \R)$. Then, by Proposition VI.2.1a) in \cite{JS}, there exists a sequence $s^n \to  s$ such that $\Delta \gamma^n(s^n) \to \Delta \gamma ( s).$ So there exists some $n_0$ such that $ s^n \in (t- \epsilon, t+ \epsilon) $  and  $ | \Delta \gamma^n (s^n ) |\geq a  $ for all $ n \geq n_0. $ In particular,  $\gamma^n \in B^\epsilon_a$ for all $n \geq n_0,$ whence the claim.   

Clearly, $\{ \gamma \in D(\R_+, \R) : |\Delta \gamma(t)| > a \} \subset B^\epsilon_a,$ which implies that $E \subset \{ \mu \in \cP ( D(\R_+, \R)) : \mu(B^\epsilon_a) > b\} \eqqcolon \mathcal{P}^\epsilon_{a,b} ,$ where the latter set is open in $\mathcal{P}(D(\R_+, \R))$ (endowed with the topology of weak convergence). 
Indeed, let $\pi \in \mathcal{P}^\epsilon_{a,b}$ and consider any sequence  $(\pi_n)_n\subset \mathcal{P}(D(\R_+, \R))$ weakly converging to $\pi$. Then the Portmanteau theorem implies that, for the open set $B^\epsilon_a \subset D(\R_+, \R)$, $\liminf_{n}\pi_n(B^\epsilon_a ) \geq \pi(B^\epsilon_a ) > b$, where the last inequality follows since $ \pi \in \mathcal{P}^\epsilon_{a,b}. $ Therefore, $\liminf_{n}\pi_n(B^\epsilon_a ) > b$, implying  that there exists $n_0 >0$ s.t. $\pi_n \in \mathcal{P}^\epsilon_{a,b} $ for all $n \geq n_0$, hence the thesis.

Recall that  $ P^N = \mathcal L ( \mu^N) $  and $P = \mathcal L ( \mu). $ Since $(P^N)_N$ converges weakly to $P , $ applying once more the Portmanteau theorem, 
\[
	\liminf_{N\to +\infty} P^N(\mathcal{P}^\epsilon_{a,b}) \geq P(\mathcal{P}^\epsilon_{a,b}),
\]
and the fact that $E \subset \mathcal{P}^\epsilon_{a,b}$ established above allows to conclude that 
\[
	\liminf_{N\to +\infty} P^N(\mathcal{P}^\epsilon_{a,b}) \geq P(\mathcal{P}^\epsilon_{a,b}) = \P(\{ \mu \, : \,  \mu \in \mathcal{P}^\epsilon_{a,b} \} )\geq \P(E) > 0 .
\]
We stress that $\P(E)$ does not depend on $\epsilon$, by definition of $E$.

In what follows, we will prove that, for some constant $C$ that does not depend on $N$ nor on $\epsilon , $
\begin{equation}\label{eq:contrad}
	\limsup_{N\to+\infty}  P^N(\mathcal{P}^\epsilon_{a,b}) \leq C \epsilon , 
\end{equation}
for all $N.$ 

Once this is proven, it is easily seen that this yields a contraction. 
Indeed, if this holds, then for any $\epsilon>0$, 
\[
	0 < \P(E) \leq \liminf_{N\to +\infty} P^N(\mathcal{P}^\epsilon_{a,b}) \leq \limsup_{N\to +\infty} P^N(\mathcal{P}^\epsilon_{a,b}) \leq C \epsilon. 
\]
Letting $ \epsilon \to 0, $ this yields a contraction.  

It remains to show that \eqref{eq:contrad} holds. Notice that 
$$\mu^N(B^\epsilon_a) = \frac{1}{N}\sum_{i=1}^N \indiq_{\{ \sup_{s\in (t-\epsilon, t+\epsilon)} |\Delta X^{N,i}_s| > a \}}.$$ 
So, by Markov's inequality and using exchangeability, 
\begin{multline}\label{eq:contrad_step1}
		P^N(\mathcal{P}^\epsilon_{a,b})  = \P(\mu^N(B^\epsilon_a) > b) \leq \frac{\E[\mu^N(B^\epsilon_a)]}{b}\\
		 = \frac{\E[\indiq_{\{ \sup_{s\in (t-\epsilon, t+\epsilon)} |\Delta X^{N,1}_s| > a  \} }]}{b} = \frac{\P(\sup_{s\in (t-\epsilon, t+\epsilon)} |\Delta X^{N,1}_s| > a )}{b} .
\end{multline}
Recall the notation introduced in \eqref{eq:decxni} above. In what follows, we call the jumps of $ I^{N,1}$ the main jumps of particle $1$ and the jumps of $J^{N}  - E^{N,1}$ the collateral jumps of it. Clearly, the random variable $ \sup_{s\in (t-\epsilon, t+\epsilon)} |\Delta X^{N,1}_s| $ can be larger than $a$ only if either the largest of the (finite number of) main jumps during $(t-\epsilon, t+\epsilon)$ exceeds $a$ or if the largest of (the finite number of) its collateral jumps during that same interval exceeds $a. $ 
So
\begin{multline}\label{eq:contrad_step2}
	\P(\sup_{s\in (t-\epsilon, t+\epsilon)} |\Delta X^{N,1}_s|   > a) \leq \P ( \sup_{s\in (t-\epsilon, t+\epsilon)} |\Delta I^{N,1}_s| > a) \\
	+ \P ( \sup_{s\in (t-\epsilon, t+\epsilon)} |\Delta (J^{N} - E^{N, 1 })_s| > a  ) \eqqcolon p_1 + p_2.  
\end{multline}

Notice that the number of main jumps of $X^{N,1}$ during $(t-\epsilon, t+\epsilon)$  is stochastically dominated by a Poisson random variable with parameter $\norme{f}_\infty 2 \epsilon .$ So we can upper bound $p_1$ by the probability that $X^{N,1}$ has at least one main jump (regardless of its size) during $(t-\epsilon, t+\epsilon)$ and this latter by the probability that the dominating Poisson random variable equals at least $1,$ which yields 
\begin{equation}\label{eq:contrad_Pmj}
	p_1  \leq 1 - e^{2\epsilon \norme{f}_\infty} \leq 2\epsilon \norme{f}_\infty. 
\end{equation}

To control $p_2$ we start by observing that the number of collateral jumps of any particle in $(t-\epsilon, t+\epsilon)$ is stochastically dominated by a Poisson random variable having parameter $2\epsilon N\norme{f}_\infty,$ that is, by the number of jumps of $\tilde P_{N\cdot}$ during an interval $[0, 2\epsilon \norme{f}_\infty]$, where $\tilde P_\cdot$ is a Poisson process of unit rate which is independent of the i.i.d. sequence of jump heights $ (U_n)_{ n \geq 1}, $ with $ U_n \sim \nu.$ So, writing 
$$ \hat J^N_s = \sum_{k=1}^{\tilde P_{N s}} \frac{U_k}{N^{1/\alpha}} ,$$ 
we have that 
$$ 
	p_2 \leq \P \left( \sup_{s \in [0, 2 \epsilon \norme{f}_\infty]} \left| \Delta \hat J^N_s\right| > a \right)  = \E\left[  \indiq_{\left\{\sup_{s \in [0, 2 \epsilon \norme{f}_\infty]} \left| \Delta \hat J^N_s\right| > a   \right\}}\right].
$$
Let  $\phi_a(x)$ be a continuous bounded function which is $0$ for $|x|< a/2$ and $1$ for $|x|\geq a$, and takes values in $(0,1)$ in between. Then
$$ 
	\E\left[  \indiq_{\left\{ \sup_{s \in [0, 2 \epsilon \norme{f}_\infty]} \left| \Delta \hat J^N_s \right| > a  \right\}}\right] \leq  \E\left[  \phi_a\left( \sup_{s \in [0, 2 \epsilon \norme{f}_\infty]} \left|  \Delta \hat J^N_s \right| \right) \right] .
$$
We know, recalling \eqref{eq:process_converging_to_Salpha}, that we have convergence in law of  $ (\hat J^N_t)_t \to (S^\alpha_t)_t. $ Now, since $D(\R_+, \R) \ni \gamma \mapsto \sup_{s\leq T} |\Delta \gamma_s |$ is a continuous function at any $\gamma$ such that $T$ is not a point of discontinuity of $\gamma$ (see Proposition VI.2.4 in \cite{JS}), and since $S^\alpha$ is a.s. continuous at $T= 2\epsilon \norme{f}_\infty$, this implies that 
$$
	\limsup_{N\to +\infty} p_2 \leq \lim_{N\to +\infty}\E\left[  \phi_a\left( \sup_{s \in [0, 2 \epsilon \norme{f}_\infty]} \left| \Delta \hat J^N_s \right| \right) \right] =  \E\left[ \phi_a\left( \sup_{s \in [0, 2 \epsilon \norme{f}_\infty]} \left| \Delta S^\alpha_s \right| \right) \right] .
$$
Since, by definition of $\phi_a$, $\phi_a(x) \leq \indiq_{\{x\geq a/2\}}$ for all $x \geq 0,$ we upper bound further 
\begin{equation*}
	\begin{split}
		\limsup_{N\to +\infty} p_2 & \leq \P\left(  \sup_{s \in [0, 2 \epsilon \norme{f}_\infty]} \left| \Delta S^\alpha_s \right| \geq a/2 \right) \\
		& = \P \left( \int_{[0, 2\epsilon\norme{f}_\infty]} \int_{|z| \geq a/2} |z| M(ds,dz) \geq 1 \right) .
	\end{split}
\end{equation*}
Taking $q<\alpha$ in the case $\alpha<1$ and $q=1$ in the case $\alpha>1$, by Markov's inequality and sub-additivity in the case $ q < 1, $ 
\begin{equation*}
	\begin{split}
	& \P \left( \int_{[0, 2\epsilon\norme{f}_\infty]} \int_{|z| \geq a/2} |z| M(ds,dz) \geq 1 \right) = \P \left( \left(\int_{[0, 2\epsilon\norme{f}_\infty]} \int_{|z| \geq a/2} |z| M(ds,dz)\right)^q \geq 1 \right) \\
	& \leq \E\left[ \left(\int_{[0, 2\epsilon\norme{f}_\infty]} \int_{|z| \geq a/2} |z| M(ds,dz)\right)^q  \right] \leq \E \left[\int_{[0, 2\epsilon\norme{f}_\infty]} \int_{|z| \geq a/2} |z|^q M(ds,dz) \right]  =  C_q \epsilon  ,
	\end{split}
\end{equation*} 
for some constant $C_q $ that does not depend on $ \epsilon.$ As a consequence, 
\begin{equation}\label{eq:contrad_Pcj}
	\limsup_{N\to +\infty} p_2  \leq C_q \epsilon .
\end{equation}

Putting together \eqref{eq:contrad_step1}, \eqref{eq:contrad_step2}, \eqref{eq:contrad_Pmj} and \eqref{eq:contrad_Pcj}, we obtain \eqref{eq:contrad}, which concludes the proof. 

\end{proof}

\subsection{Auxiliary lemmas}
Recall that $D(\R_+, \R)$ is endowed with the Skorokhod topology, and that $C(\R_+, \R)$ and $ \Lambda$ are endowed with the topology of locally uniform convergence. Product spaces will always be endowed with the product topology. 

\begin{lemma}\label{lem:continuity_of_composition}
The map $\Phi : D(\R_+, \R) \times \Lambda^{ \lambda_1, \lambda_2}  \to D(\R_+, \R), \, (\gamma, \eta) \mapsto \gamma \circ \eta$, that applies a time change of $\Lambda^{\lambda_1, \lambda_2}$ to a c{\`a}dl{\`a}g path, is continuous. 
\end{lemma}

\begin{proof}
Consider any sequence $(\gamma^N, \eta^N)_N \subset D(\R_+, \R) \times \Lambda^{ \lambda_1, \lambda_2}$ such that $(\gamma^N, \eta^N)_N \to (\gamma,\eta) \in D(\R_+, \R) \times \Lambda^{ \lambda_1, \lambda_2} $. 
In particular, by Theorem VI.1.14(a) of \cite{JS}, there exists a sequence $(\lambda^N)_N \subset \Lambda$ of time changes such that 
$\sup_{s \leq T } | \gamma^N( \lambda^N(s)) - \gamma(s)| \to 0$ for all $T >0$  and $\sup_{s \leq T}|  \lambda^N(s) - s|  \to 0$ for all $T>0$. 

We want to show that, as $N\to +\infty$, $ \gamma^N\circ \eta^N$ converges to $ \gamma\circ\eta  $ in the Skorokhod topology.   
By Lemma VI.1.44 in \cite{JS}, this reduces to find a sequence of time changes $(\tilde\lambda^N)_N \subset \Lambda$ such that $\sup_{s\leq T} | \tilde \lambda^N(s) - s | \to 0$ for all $T>0$ as $N\to +\infty$, and $\sup_{s\leq T} |\gamma^N \circ \eta^N \circ \tilde \lambda^N(s) - \gamma\circ \eta(s) | \to 0$ for all $T>0$. 

In what follows, we check that this holds for the sequence $\tilde \lambda^N \coloneqq (\eta^N)^{-1} \circ \lambda^N \circ \eta$. 
Indeed, $\tilde\lambda^N\in \Lambda$ for each $N$ since $\Lambda$ is closed under inversion and composition, and $\gamma^N \circ \eta^N \circ \tilde \lambda^N = \gamma^N \circ \lambda^N \circ \eta \to \gamma \circ \eta$ in the sup norm over any compact interval, since  $\sup_{s \leq T } | \gamma^N( \lambda^N( \eta (s))) - \gamma(\eta(s))| \le \sup_{ s \le \lambda_2 T }  | \gamma^N( \lambda^N( s)) - \gamma(s)| \to 0$ for all $T>0$.


On the other hand, 
\begin{align}\label{eq:tildelambdaN_step1}
	|\tilde \lambda^N(t) - t| & = |(\eta^N)^{-1} \circ \lambda^N (\eta(t)) - t | = |(\eta^N)^{-1} \circ \lambda^N (\eta(t)) - (\eta^N)^{-1} \circ \eta^N(t) | .
\end{align}
Since  $\eta^N \in \Lambda^{\lambda_1, \lambda_2}$, we have for all $0 \le  s<t  < \infty, $ 
\[
	\frac{t-s}{\eta^N(t) - \eta^N(s)} \in \left[ \frac{1}{\lambda_2}, \frac{1}{\lambda_1} \right]  .
\]
In particular, choosing $t = (\eta^N)^{-1}(x) $ and $s =   (\eta^N)^{-1}(y) $ for $y < x, $ 
\[	
	 (\eta^N)^{-1}(x) -  (\eta^N)^{-1}(y) \leq  \frac{1}{\lambda_1} (x-y) \mbox{ for all } 0 \le  y<x < \infty  .
\]
Therefore, coming back to \eqref{eq:tildelambdaN_step1}, 
\begin{align*}
	|\tilde \lambda^N(t) - t| & \leq  \frac{1}{\lambda_1} | \lambda^N (\eta(t)) - \eta^N(t) | \leq  \frac{1}{\lambda_1} \left( | \lambda^N (\eta(t)) - \eta(t) | + | \eta(t) - \eta^N(t) | \right) ,
\end{align*}
and, taking the supremum over $[0,T]$ on both sides, we obtain
\begin{multline*}
	\sup_{t\leq T}|\tilde \lambda^N(t) - t|  \leq  \frac{1}{\lambda_1}  \sup_{t\leq T} | \lambda^N (\eta(t)) - \eta(t) | +  \frac{1}{\lambda_1}  \sup_{t\leq T} | \eta(t) - \eta^N(t) |  \\
	\leq \frac{1}{\lambda_1}  \sup_{t\leq \lambda_2 T} | \lambda^N (t)- t | +  \frac{1}{\lambda_1}  \sup_{t\leq T} | \eta(t) - \eta^N(t) | ,
\end{multline*}
and this converges to zero for any $T>0$ due to uniform convergence on all compacts of $\lambda^N$ to $id$ and of $\eta^N$ to $\eta.$ 
\end{proof}

\begin{lemma}\label{lem:app2}
The map 
\[
	\mathcal{P}(D(\R_+, \R)) \to \Lambda^{\underline f , \norme{f}_\infty}  , \, m \mapsto \int_0^\cdot m_s(f) ds 
\]
is continuous at any point $ \tilde m $ for which  $\tilde m (\{ \gamma \in D(\R_+, \R) \, | \, \gamma(s) \neq \gamma(s^-)\})=0,$ for any fixed $ s\geq 0.$ 
\end{lemma} 

\begin{proof}
By Chapter 13 of \cite{billingsley}, the projection map $\pi_s : m \mapsto  m_s$ is continuous at $ \tilde m. $  As a consequence, if $ m^n $ converges weakly to $ \tilde m$ in $\cP ( D(\R_+, \R)),$  then also  $ m^n_s $ converges weakly to $ \tilde m_s $ in $ \cP ( \R) .$ Since $f$ is continuous and bounded, the definition of weak convergence implies that   $m^n_s ( f) \to \tilde m_s ( f).$ The assertion then follows from dominated convergence. 
\end{proof}

\begin{lemma}\label{lem:app4}
Let $g:\R^2\to \R$ be continuous and bounded and let $ \nu $ be a probability measure on $ \R^\ast .$ 
Let $(\gamma, m )\in D(\R_+, \R) \times  \mathcal{P}(D(\R_+, \R)) $ and denote by  $(\gamma_s, m_s)$ its projection on 
$\R\times \mathcal{P}(\R)$. Assume moreover that  for any fixed $t \in \R_+$,  
	\[ 
		m(\{ \gamma \in D(\R_+, \R) \, : \, |\Delta \gamma(t)| \neq   0 \}) = 0  .\]
Suppose that $(\gamma^N, m^N )\to (\gamma, m)$ in $D(\R_+, \R) \times  \mathcal{P}(D(\R_+, \R))$, where $\mathcal{P}(D(\R_+, \R))$ is endowed with the topology of weak convergence. 
Then we have the convergence in $C(\R_+, \R)$  
\begin{equation*}
	 \int_0^\cdot f(\gamma^N_s) \int_{\R^\ast} \nu(du)  g(\psi(\gamma^N_s, m^N_s),  u/N^{1/\alpha}) ds  \to  \int_0^\cdot f(\gamma_s)   g(\psi(\gamma_s, m_s),  0) ds .
\end{equation*}
\end{lemma}

\begin{proof}
We have joint continuity of the projections $\pi_s : D(\R_+, \R)\times \mathcal{P}(D(\R_+, \R)) \to \R\times \mathcal{P}(\R), (\gamma, m) \mapsto (\gamma_s, m_s)$ for Lebesgue almost all $ s \geq 0, $ at $(\gamma, m)$ for any $\gamma\in D(\R_+, \R).$ 
This follows since we have the product topology on the spaces $D(\R_+, \R)\times \mathcal{P}(D(\R_+, \R))$, $\R\times \mathcal{P}(\R)$ and since
\begin{enumerate}
	\item for any $s$ except a set of zero Lebesgue measure, $\pi_s$ is continuous in the first component at any $\gamma\in D(\R_+, \R)$ (indeed, the mapping $\gamma \mapsto \gamma_s$ is continuous at $\gamma$ if and only if $\gamma$ is continuous at $s$, see e.g. \cite{JS} VI.2.3, and $\gamma $ possesses at most a countable number of jumps);
	\item for any fixed $s,$ $\pi_s$ is continuous in the second component at $m$ by assumption on $m$ (see Chapter 13 of \cite{billingsley}).
\end{enumerate}
Since $f$, $\psi$ and $g$ are continuous and bounded, recalling Assumption \ref{ass:weak_cont},  we have that  $f(\gamma^N_s)g(\psi(\gamma^N_s, m^N_s), u/N^{1/\alpha})\to f(\gamma_s) g(\psi(\gamma_s,m_s),  0)$ for fixed $u$ and for Lebesgue almost all $ s.$ 
The proof is concluded using dominated convergence.
\end{proof}

\begin{lemma}\label{lem:app5}
\[
	 (\gamma, m) \mapsto \int_0^\cdot b(\gamma_s, m_s) ds 
\]
is continuous from $D(\R_+, \R)\times \mathcal{P}(D(\R_+, \R))$ to $C(\R_+, \R),$ at any point $ ( \gamma , m ) $ such that for any fixed $t \in \R_+$,  
	\[ 
		m(\{ \gamma \in D(\R_+, \R) \, : \, |\Delta \gamma(t)| \neq  0 \}) = 0  . 
	\]
\end{lemma}

\begin{proof}
The assertion follows using the same arguments as those of the proof of Lemma \ref{lem:app4}. 
\end{proof}

\begin{lemma}\label{lem:mut}
For any bounded and continuous real function $f$ and any probability measure $\mu$ on $(D(\R_+, \R),\D ),$ the function $\R_+\to\R,$ given by $t\to\mu_t(f)$ is c{\`a}dl{\`a}g.
\end{lemma}
\begin{proof}
Let $t_n \downarrow t$ and $ \gamma \in D(\R_+, \R).$  Since $\gamma$ is right-continuous, $\gamma(t_n)\to \gamma(t).$ Since $f$ is continuous, $f(\gamma_{t_n})\to f(\gamma_t).$ In other words, $f\circ\pi_{t_n}(\gamma)\to f\circ\pi_t(\gamma)$ for all $\gamma\in D(\R_+, \R).$ Since $f$ is bounded, and $\mu(D(\R_+, \R))<\infty,$
dominated convergence gives
$$\mu_{t_n}(f) \coloneqq \int_{D(\R_+, \R)}f\circ\pi_{t_n}(\gamma)d\mu\to\int_{D(\R_+, \R)}f\circ\pi_{t}(\gamma)d\mu \coloneqq \mu_t(f).$$
Let now $t_n\uparrow t.$ Since $\gamma$ admits left limits, $\gamma(t_n)\to \gamma(t-),$ where $\gamma(t-) \coloneqq \lim_{t_n\uparrow t}\gamma(t_n),$
such that, by continuity, $f(\gamma( t_n)) \to f(\gamma(t-)).$ But as a limit of measurable applications, $\lim_{t_n\uparrow t}\pi_{t_n}$ is measurable, hence $f\circ \lim_{t_n\uparrow t}\pi_{t_n}$ is  $(\D,\B(\R))-$measurable, and again by dominated convergence for bounded $f,$
$$\mu_{t_n}(f)=\int_{D(\R_+, \R)}f\circ\pi_{t_n}(\gamma)d\mu\to\int_{D(\R_+, \R)}f\circ(\lim_{t_n\uparrow t}\pi_{t_n})(\gamma)d\mu \coloneqq \mu_{t-}(f).$$ 
\end{proof}

{\noindent\textbf{Acknowledgements}:
Eva L\"ocherbach has received funding from the European Union's Horizon Europe research and innovation programme under the Marie Sklodowska-Curie Actions Staff Exchanges (Grant Agreement No. 101183168, Call: HORIZON-MSCA-2023-SE-01). 
E.M. is member of the Gruppo Nazionale per l'Analisi Matematica, la Probabilit{\`a} e le loro Applicazioni (GNAMPA) of Istituto Nazionale di Alta Matematica (INdAM), and she was partially supported by the project \lq\lq Ferromagnetism versus synchronization: how does disorder destroy universality?\rq\rq; E.M. also acknowledges financial support from the ANR grant ANR-21-CE40-0006 SINEQ and from Fondazione \lq\lq Ing. Aldo Gini\rq\rq. 
}

{\noindent\textbf{Disclaimer}: Funded by the European Union. Views and opinions expressed are however those of the author(s) only and do not necessarily reflect those of the European Union or the European Education and Culture Executive Agency (EACEA). Neither the European Union nor EACEA can be held responsible for them.}

\end{document}